\documentclass[11pt,english]{amsart}

\usepackage[T1]{fontenc}
\usepackage[latin9]{inputenc}
\usepackage[a4paper]{geometry}
\usepackage{babel}
\usepackage{amstext}
\usepackage{amsthm}
\usepackage{amssymb}
\usepackage[all]{xy}
\usepackage[unicode=true,pdfusetitle,
 bookmarks=true,bookmarksnumbered=false,bookmarksopen=false,
 breaklinks=true,pdfborder={0 0 1},backref=false,colorlinks=false]
 {hyperref}
\hypersetup{
 linktocpage=true}

\makeatletter
\numberwithin{equation}{section}
\numberwithin{figure}{section}
\theoremstyle{plain}
\newtheorem{thm}{\protect\theoremname}[section]
\theoremstyle{definition}
\newtheorem{defn}[thm]{\protect\definitionname}
\theoremstyle{plain}
\newtheorem{lem}[thm]{\protect\lemmaname}
\theoremstyle{plain}
\newtheorem{prop}[thm]{\protect\propositionname}
\theoremstyle{definition}
\newtheorem{example}[thm]{\protect\examplename}
\theoremstyle{remark}
\newtheorem{rem}[thm]{\protect\remarkname}
\theoremstyle{plain}
\newtheorem{cor}[thm]{\protect\corollaryname}

\usepackage{tikz-cd}
\usepackage{tikz}
\usepackage{mathabx}
\usepackage[font=small,labelfont=bf]{caption}
\usepackage{graphicx}
\usepackage{amsfonts, amsmath, amsthm, amssymb}
\allowdisplaybreaks
\usepackage{sagetex}

\usepackage{caption,subcaption,tikz}

\tikzset{%
    add/.style args={#1 and #2}{
        to path={%
 ($(\tikztostart)!-#1!(\tikztotarget)$)--($(\tikztotarget)!-#2!(\tikztostart)$)%
  \tikztonodes},add/.default={.2 and .2}}
}  
\usepackage{dynkin-diagrams}
\usepackage{sagetex}

\newcommand{\xyR}[1]{
  \xydef@\xymatrixrowsep@{#1}}
\newcommand{\xyC}[1]{
  \xydef@\xymatrixcolsep@{#1}}
\theoremstyle{definition}

\newtheorem{conj}[thm]{Conjecture}

\usetikzlibrary{shapes.geometric, arrows}
\usetikzlibrary{arrows}
\usetikzlibrary{intersections}
\usetikzlibrary{shapes.misc}
\usetikzlibrary{patterns}
\usetikzlibrary{calc,patterns,angles,quotes}
\usetikzlibrary{positioning}
\usetikzlibrary{decorations.markings,decorations.pathreplacing}

\tikzset{cross/.style={cross out, draw, 
         minimum size=2*(#1-\pgflinewidth), 
         inner sep=0pt, outer sep=0pt}}

\definecolor{redviolet}{rgb}{0.78,0.08,0.52}
\xdefinecolor{darkgreen}{RGB}{0, 100, 0}

\makeatother

\providecommand{\corollaryname}{Corollary}
\providecommand{\definitionname}{Definition}
\providecommand{\examplename}{Example}
\providecommand{\lemmaname}{Lemma}
\providecommand{\propositionname}{Proposition}
\providecommand{\remarkname}{Remark}
\providecommand{\theoremname}{Theorem}

\begin{document}
\title{Weyl groups and cluster structures of families of log Calabi-Yau surfaces}
\author{Yan Zhou}
\begin{abstract}
Given a generic Looijenga pair $(Y,D)$ together with a toric model
$\rho:(Y,D)\rightarrow(\overline{Y},\overline{D})$, one can construct
a seed ${\bf s}$ such that the corresponding $\mathcal{X}$-cluster
variety $\mathcal{X}_{{\bf s}}$ can be viewed as the universal family
of the log Calabi-Yau surface $U=Y\setminus D$. In cases where $(Y,D)$
is positive and $\mathcal{X}_{{\bf s}}$ is \emph{not} acyclic, we
describe the action of the Weyl group of $(Y,D)$ on the scattering
diagram $\mathfrak{D}_{{\bf s}}$. Moreover, we show that there is
a Weyl group element ${\bf w}$ of order $2$ that either agrees with
or approximates the Donaldson-Thomas transformation ${\rm DT}_{\mathcal{X}_{{\bf s}}}$
of $\mathcal{X}_{{\bf s}}$. As a corollary, ${\rm DT}_{\mathcal{X}_{{\bf s}}}$
is cluster. In positive non-acyclic cases, we also apply the folding
technique as developed in \cite{YZ} and construct a maximally folded
new seed $\overline{{\bf s}}$ from ${\bf s}$. The $\mathcal{X}$-cluster
variety $\mathcal{X}_{\overline{{\bf s}}}$ is a locally closed subvariety
of $\mathcal{X}_{{\bf s}}$ and corresponds to the maximally degenerate
subfamily in the universal family. We show that the action of the
special Weyl group element ${\bf w}$ on $\mathfrak{D}_{{\bf s}}$
descends to $\mathfrak{D}_{\overline{{\bf s}}}$ and permutes distinct
subfans in $\mathfrak{D}_{\overline{{\bf s}}}$ , generalizing the
well-known case of the Markov quiver.
\end{abstract}

\maketitle

\section{Introduction}

\subsection{Background.}

In the 90's, Lusztig did seminal work on total positivity in reductive
groups and canonical bases in quantum groups (\cite{Lu90,Lu94}).
\emph{Cluster algebras} (\cite{FZ02}) were invented by Fomin and
Zelevinsky in 2000 as an algebraic and combinatorial framework to
study total positivity and Lusztig's\emph{ dual canonical bases} in
semisimple Lie groups. Remarkably, cluster structures also appear
in many seemingly unrelated areas in mathematics like Teichmüller
theory, Poisson geometry, quiver representations, quantum field theory,
discrete dynamical systems, combinatorics..., and thus serve as potential
bridges between these various areas. 

In 2014, Gross, Hacking, Keel and Kontsevich in \cite{GHKK18} proposed
a new framework to understand cluster algebras via\emph{ mirror symmetry}.
Mirror symmetry originated from string theory and reveals surprising
duality between complex geometry and symplectic geometry on pairs
of \emph{Calabi-Yau varieties} that are mirror dual to each other.
Calabi-Yau varieties are complex manifolds with a holomorphic volume
form that is unique up to scaling. \emph{Cluster varieties}, varieties
obtained by gluing algebraic tori together via a special class of
birational maps called \emph{mutations} which are encoded by cluster
combinatorics, are examples of open Calabi-Yau varieties called \emph{log
Calabi-Yau varieties}.

Using combinatorial gadgets from enumerative geometry like \emph{scattering
diagrams} and\emph{ broken lines}, Gross, Hacking, Keel and Kontsevich
construct \emph{canonical bases} for cluster algebras consisting of
elements called \emph{theta functions}. This new framework suggests
that cluster varieties from representation theory, like double Bruhat
cells, Grassmanians, base affine spaces, should give canonical bases
of representations intrinsic to the geometry of their mirror dual
varieties. Indeed, the geometric realization of broken lines and theta
functions are done in \cite{KY} for smooth log Calabi-Yau varieties
containing a Zariski open dense torus, which include the affine closure
of double Bruhat cells. 

As we expect that canonical bases are intrinsic to the geometry of
log Calabi-Yau varieties and should be independent of cluster structures,
a natural question arises: Does there exist a cluster variety with
two non-equivalent cluster structures, and yet either one yields the
same canonical basis? For precisely what we mean by `non-equivalent'
cluster structures, see Definition \ref{def:non-eqiuv-cluster}. In
\cite{YZ}, we made a first step towards a positive answer to such
question - Yes, there exists a cluster variety with two non-equivalent
cluster structures. 

The idea of \cite{YZ} is to use distinct subfans in scattering diagrams
to construct non-equivalent cluster structures. The standard cluster
structure of a cluster variety corresponds to a subfan structure in
the scattering diagram called the \emph{cluster complex}. When the
\emph{Donaldson-Thomas transformation (DT)} \cite{GS16} of a cluster
variety is rational but not cluster, i.e., not a composition of a
sequence of cluster mutations, the scattering diagram will have at
least two distinct subfans and they will provide, at the very least,
two non-equivalent cluster atlases (Definition \ref{def:non-eqiuv-cluster}).
If we further show that these two cluster atlases agree up to codimension,
then we get an example of a cluster variety with two-non equivalent
cluster structures. 

The goal of this paper is to construct a class of cluster varieties
with non-equivalent cluster atlases generalizing the example we give
in \cite{YZ}, using the geometry of Looijenga pairs \cite{L81} and
the \emph{folding} techniques we developed in \cite{YZ} in the context
of scattering diagrams. 

\subsection{Looijenga pairs and rank $2$ cluster varieties. }

A \emph{Looijenga pair} $(Y,D)$ is a smooth projective surface $Y$
together with a connected nodal singular curve $D\in|-K_{Y}|$. By
the adjunction formula, $p_{a}(D)=1$. So $D$ is either an irreducible
rational curve with a single node or a cycle of smooth rational curves.
Since $K_{Y}+D$ is trivial, the complement $U=Y\setminus D$ is a
log Calabi-Yau surface. Write $D=D_{1}+\cdots+D_{r}$ where $D_{i}\,(1\leq i\leq r)$
is irreducible. Let 
\[
D^{\perp}:=\left\{ \alpha\in{\rm Pic}(Y)\mid\alpha\cdot\left[D_{i}\right]=0\text{ for all }i\right\} 
\]
and
\[
T_{(D^{\perp})^{*}}:=\left(D^{\perp}\right)^{*}\otimes\mathbb{C}^{*}.
\]
By the Global Torelli Theorem for Looijenga pairs (cf. Theorem 1.8
of \cite{GHK12}), we could view $T_{(D^{\perp})^{*}}$ as the \emph{period
domain} of the Looijenga pair $(Y,D)$. We say $(Y,D)$ is \emph{positive}
if the intersection matrix given by $(D_{i}\cdot D_{j})$ is\emph{
not }negative semidefinite. We say a Looijenga pair $(Y,D)$ is \emph{generic}
if $Y$ does not contain a smooth rational curve with self-intersection
$-2$ and disjoint from $D$. In this paper, unless we say otherwise,
we assume that our Looijenga pair is generic. We say a birational
morphism $\rho:(Y,D)\rightarrow(\overline{Y},\overline{D})$ is a
\emph{toric model }for $(Y,D)$ if $\overline{Y}$ a smooth projective
toric surface and $D$ is the strict transform of the toric boundary
$\overline{D}$ of $\overline{Y}$. 

By \emph{rank $2$ cluster varieties}, we mean cluster varieties whose
seeds have exchange matrices of rank $2$. In \cite{GHK15}, the relationship
between Looijenga pairs and rank $2$ cluster varieties is explored.
Given a generic Looijenga pair $(Y,D)$ together with a toric model
$\rho:(Y,D)\rightarrow(\overline{Y},\overline{D})$, one can construct
a skew-symmetric cluster seed ${\bf s}$ using the combinatorial data
of the exceptional curves of the toric model. Let $K$ be the kernel
of the skew-symmetric form of ${\bf s}$. This kernel $K$ has codimension
$2$ and therefore the corresponding $\mathcal{X}$-cluster variety
$\mathcal{X}_{{\bf s}}$ is a rank $2$ cluster variety. By Theorem
5.5 of \cite{GHK15}, there is an isomorphism $K\simeq D^{\perp}$.
The space $\mathcal{X}_{{\bf s}}$ fibers over $T_{K^{*}}\simeq T_{(D^{\perp})^{*}}$
and the generic fibers in the family $\mathcal{X}_{{\bf s}}\rightarrow T_{(D^{\perp})^{*}}$,
up to codimension $2$, are log Calabi-Yau surfaces deformation equivalent
to $U:=Y\setminus D$ (see subsection \ref{subsec:Relationship-to-cluster}
and section 5 of \cite{GHK15} for more details). Moreover, we can
view $\mathcal{X}_{{\bf s}}$ as the universal family of $U$ (cf.
Theorem \ref{thm:rank 2 clus. var. and log CY surf.} and Theorem
5.5 of \cite{GHK15}). This relationship goes the other way, that
is, we could start with a skew-symmetric seed ${\bf s}$ of rank $2$
and construct the corresponding Looijenga pair. To get rid of the
skew symmetric condition, we could allow singularities for our Looijenga
pairs. Given these established relationships between Looijenga pairs
and rank $2$ cluster varieties, the main theme of this paper is to
understand the theory of rank $2$ cluster varieties from the geometry
of Looijenga pairs. As we will see later, the very first application
along this line is the well-known example of cluster varieties associated
to the \emph{Markov quiver}, the main example we explored in \cite{YZ}. 

Since the theory of acyclic cluster varieties is well-understood,
in this paper, we will focus on the cases where the corresponding
rank $2$ cluster variety associated to a Looijenga pair $(Y,D)$
is \emph{not} acyclic, or equivalently, where $D^{\perp}\simeq{\bf D}_{n}\,(n\geq4)$
or ${\bf E}_{n}\,(n=6,7,8)$ (cf. Proposition \ref{prop:classification rank 2 clus. var.}).
In \emph{positive non-acyclic} cases, we also denote the corresponding
$\mathcal{X}$-cluster variety by $\mathcal{X}_{D^{\perp}}$ when
we want to emphasize the particular deformation type of $U=Y\setminus D$. 

\subsection{Main results.}

Let $\mathfrak{D}_{{\bf s}}$ be the scattering diagram associated
to a seed ${\bf s}$ constructed from a toric model for a positive,
non-acyclic Looijenga pair. A natural question arises: Does the action
of the Weyl group of $(Y,D)$ on $\mathcal{X}_{{\bf s}}$ preserve
the cluster complex in $\mathfrak{D}_{{\bf s}}$? The answer is positive
(Proposition \ref{prop:Weyl group preseves cluster complex}): The
Weyl group acts faithfully on the scattering diagram as a subgroup
of the cluster modular group and in particular, preserves the cluster
complex.

The action of the Weyl group on $\mathfrak{D}_{{\bf s}}$ enables
us to have a simple, geometric understanding of the Donaldson-Thomas
transformation\emph{ }${\rm DT}_{\mathcal{X}_{D^{\perp}}}$ of $\mathcal{X}_{D^{\perp}}$
(for details of the Donaldson-Thomas transformations of cluster varieties,
see \cite{GS16}). In particular, in subsection \ref{subsec:Factorizations-of-DT-transformat}
and subsection \ref{subsec:Factorizations-of-DT-transformat-1}, we
show that there exists a special element ${\bf w}$ of order $2$
in the Weyl group, such that ${\bf w}$ either agrees with or approximates
${\rm DT}$\footnote{See Remark \ref{rem:It-follows-from} for the difference between ${\bf w}$
and ${\rm DT}_{\mathcal{X}_{\mathbf{D}_{n}}}$ when $n\geq5$ and
Appendix \ref{sec:Mutation-sequences-for} for the difference between
${\bf w}$ and ${\rm DT}_{\mathcal{X}_{\mathbf{E}_{6}}}$.}. Thus, we prove the following theorem:
\begin{thm}
(Theorem \ref{Thm:The-Donaldson-Thomas-transformat}, Theorem \ref{Thm:The-Donaldson-Thomas-transformat-1})
The Donaldson-Thomas transformation of $\mathcal{X}_{D^{\perp}}$
for $D^{\perp}\simeq{\bf D}_{n}\,(n\geq4)$ or ${\bf E}_{n}\,(n=6,7,8)$
is cluster, that is, it can be factored as compositions of cluster
transformations.
\end{thm}

Another main goal of this paper is to apply the folding technique
we developed in \cite{YZ} to families of log Calabi-Yau surfaces.
Let ${\bf s}_{0}$ be a general initial seed, not necessarily constructed
from a Looijenga pair. Suppose ${\bf s}_{0}$ has certain symmetry
(cf. subsection \ref{eq: symmetry of seed}), then recall that in
\cite{YZ} we showed that there is a quotient construction called
\emph{folding} that produces a new seed $\overline{{\bf s}_{0}}$
. Moreover, this quotient construction extends to the level of scattering
diagrams and we can recover $\mathfrak{D}_{\overline{{\bf s}_{0}}}$
from $\mathfrak{D}_{{\bf s}_{0}}$. The cluster complex $\Delta_{{\bf s}_{0}}^{+}$
in $\mathfrak{D}_{{\bf s}_{0}}$, when restricted to $\mathfrak{D}_{\overline{{\bf s}_{0}}}$,
can break into distinct subfans. Denote by $\overline{\Delta_{{\bf s}_{0}}^{+}}$
the restriction of $\Delta_{{\bf s_{0}}}^{+}$ in $\mathfrak{D}_{\overline{{\bf s}_{0}}}$.
In particular, $\overline{\Delta_{{\bf s}_{0}}^{+}}$ always contains
the cluster complex $\Delta_{\overline{{\bf s}_{0}}}^{+}$. Denote
by $\mathcal{A}_{{\rm prin},{\bf s}_{0}}$ (resp. $\mathcal{A}_{{\rm prin},\overline{{\bf s}_{0}}}$
) the cluster variety of $\mathcal{A}_{{\rm prin}}$-type associated
to ${\bf s}_{0}$ (resp. $\overline{{\bf s}_{0}}$). If two subfans
in $\overline{\Delta_{{\bf s}_{0}}^{+}}$ are distinct, i.e., they
have trivial intersection, the wall-crossing between them will be
non-cluster, but we can nevertheless build a variety $\mathcal{A}_{{\rm prin},\overline{\Delta_{{\bf s}_{0}}^{+}}}$
using all chambers in $\overline{\Delta_{{\bf s}_{0}}^{+}}$. 

In cases where $(Y,D)$ is positive non-acyclic, if ${\bf s}$ is
the seed constructed from $(Y,D)$, there is a way to maximally fold
${\bf s}$ to a new seed $\overline{{\bf s}}$. The $\mathcal{X}$-cluster
variety $\mathcal{X}_{\overline{{\bf s}}}$ corresponds to the maximally
degenerate subfamily of the universal family of $U=Y\setminus D$
(cf. subsection \ref{subsec:Folding-and-degenerate}). The action
of the special Weyl group element ${\bf w}$ we use to factorize ${\rm DT}$
into cluster transformations restricts to an action on the scattering
diagram $\mathfrak{D}_{\overline{{\bf s}}}$ associated to $\bar{{\bf s}}$
and permutes distinct subfans in $\mathfrak{D}_{\overline{{\bf s}}}$:
\begin{thm}
\label{thm:(Theorem-)-The}(Theorem \ref{thm:The-scattering-diagram})
The scattering diagram $\mathfrak{D}_{\overline{{\bf s}}}$ has two
distinct subfans corresponding to $\Delta_{\overline{{\bf s}}}^{+}$
and $\Delta_{\overline{{\bf s}}}^{-}$ respectively. The action of
Weyl group element ${\bf w}$ on $\mathfrak{D}_{{\bf s}}$ descends
to $\mathfrak{D}_{\overline{{\bf s}}}$ and interchanges $\Delta_{\overline{{\bf s}}}^{+}$
and $\Delta_{\overline{{\bf s}}}^{-}$. In particular, if we build
a variety $\mathcal{\tilde{A}}_{{\rm prin},\bar{{\bf s}}}$ using
both $\Delta_{\overline{{\bf s}}}^{+}$ and $\Delta_{\overline{{\bf s}}}^{-}$,
then $\Delta_{\overline{{\bf s}}}^{+}$ and $\Delta_{\overline{{\bf s}}}^{-}$
provide non-equivalent atlases of cluster torus charts for $\mathcal{\tilde{A}}_{{\rm prin},\bar{{\bf s}}}$.
\end{thm}

When $D^{\perp}\simeq\mathbf{D}_{4}$, ${\bf s}$ is the seed given
by the ideal triangulation of $4$-punctured sphere and $\overline{{\bf s}}$
is the seed given by the ideal triangulation of $1$-punctured torus
and the corresponding quiver is the Markov quiver. Thus, we recover
the main example we investigated in \cite{YZ}. 

In the future, it will be interesting to further investigate whether
the two atlases provided by $\Delta_{\bar{{\bf s}}}^{\pm}$ agree
up to codimension $2$ and yield the same canonical basis. 

\subsection*{Acknowledgements.}

First I want to thank my advisor Sean Keel for suggesting the initial
direction of this project. I also want to thank Travis Mandel for
clarifying my confusion about rank $2$ cluster varieties and for
his numerous helpful comments on the initial draft of this paper.

\section{Preliminaries of Looijenga pairs\label{sec:Preliminaries-of-Looijenga}}

Throughout this paper, we work over the field $\mathbb{C}$.
\begin{defn}
A \emph{Looijenga pair} $(Y,D)$ is a smooth projective surface $Y$
with a connected singular nodal curve $D\in|-K_{Y}|$.
\end{defn}

By adjunction formula, the arithmetic genus of $D$ is $1$. So $D$
is either an irreducible genus $1$ curve with a single node or a
cycle of smooth rational curves. In this paper, we only consider cases
where $D$ is a cycle of smooth rational curves with \textbf{at least
$3$ }irreducible components\textbf{.}

Let $U=Y\setminus D$. Then $U$ is a\emph{ log Calabi-Yau surface},
that is, there exists a canonical holomorphic volume $\omega$, unique
up to scaling, such that for any normal crossing compactification
$U\subset Y^{'}$, $\omega$ has at worst a simple pole along each
irreducible component of the boundary $D^{'}=Y^{'}\setminus U$.
\begin{defn}
\label{def:Let--be}Let $(Y,D)$ be a Looijenga pair. Given a smooth
projective toric surface $\overline{Y}$ together with its toric boundary
$\overline{D}:=Y\setminus\mathbb{G}_{m}^{2}$, i.e., $\overline{D}$
is the union of toric divisors of $\overline{Y}$, we say that $(\overline{Y},\overline{D})$
is a \emph{toric model} for $\left(Y,D\right)$ if there exists a
birational morphism $\rho:Y\rightarrow\overline{Y}$ such that $\rho$
is a blowup at finite number of points in the smooth locus of $\overline{D}$
(we allow infinitely near points) and $D$ is the strict transform
of $\overline{D}$. We say $(Y',D^{'})$ is a \emph{simple toric blowup}
of $(Y,D)$ if $Y^{'}$ is the blowup of $Y$ at a nodal point of
$D$ and $D^{'}$ is the reduced inverse image of $D$. A\emph{ toric
blowup} of $(Y,D)$ is a composition of simple toric blowups.
\end{defn}

The following lemma shows that each Looijenga pair has a very elementary
construction via a toric model:
\begin{lem}
(Proposition 1.3 in \cite{GHK11}) Given a Looijenga pair $(Y,D)$,
there exists a toric blowup $(Y^{'},D^{'})$ of $(Y,D)$ such that
$(Y^{'},D^{'})$ has a toric model.
\end{lem}

Since a toric blowup does not change the log Calabi-Yau surface we
are looking at, throughout this paper, we assume that a Looijenga
pair has a toric model.
\begin{defn}
Given a Looijenga pair $(Y,D)$, an \emph{orientation} of $D$ is
an orientation of the dual graph of $D$.
\end{defn}

For the rest of this paper, we always fix an orientation of $D$ and
an ordering $D=D_{1}+D_{2}+\cdots+D_{r}(r\geq3)$ of the irreducible
components of $D$ compatible with our choice of the orientation.
\begin{defn}
Given a Looijenga pair $(Y,D)$ where $D=D_{1}+D_{2}+\cdots+D_{r}$,
we say $(Y,D)$ is\emph{ positive} if the intersection matrix $(D_{i}\cdot D_{j})$
is \emph{not }negative semidefinite.
\end{defn}

\subsection{Period mappings.}

By Lemma 2.1 of \cite{GHK12}, an orientation of $D$ canonically
identifies ${\rm Pic}^{0}(D)$ with $\mathbb{G}_{m}$.
\begin{defn}
Let 
\[
D^{\perp}:=\left\{ \alpha\in{\rm Pic}\left(Y\right)\mid\alpha\cdot\left[D_{i}\right]=0\text{ for all }i\right\} 
\]
and
\[
T_{\left(D^{\perp}\right)^{*}}:={\rm Hom}\left(D^{\perp},{\rm Pic}^{0}(D)\right)\simeq{\rm Hom}\left(D^{\perp},\mathbb{G}_{m}\right).
\]
We call the following homomorphism $\phi_{Y}\subset T_{\left(D^{\perp}\right)^{*}}$
the \emph{period point} of $(Y,D)$:
\begin{align*}
\phi_{Y}:D^{\perp} & \rightarrow{\rm Pic}^{0}(D)\simeq\mathbb{G}_{m}\\
L & \mapsto L\mid_{D}.
\end{align*}
Recall that $U=Y\setminus D$ has a canonical volume form $\omega$,
unique up to scaling, with simple poles along $D$. According to the
theory of extensions of mixed Hodge structures as developed in \cite{Ca85},
the mixed Hodge structure on $H^{2}\left(U\right)$ is classified
by periods of $\omega$ over cycles in $H_{2}(U,\mathbb{Z})$. By
Proposition 3.12 of \cite{F}, the data of periods of $\omega$ over
cycles in $H_{2}(U,\mathbb{Z})$ are equivalent to the homomorphism
$\phi_{Y}$. Hence the naming ``period point'' of $\phi_{Y}$ is
justified and $\phi_{Y}$ describes the mixed Hodge structure of $U$.
See Theorem 1.8 of \cite{GHK12} for the full statement of the Global
Torelli Theorem for Looijenga pairs. For the description of moduli
stacks of Looijenga pairs via period domains, see section 6 of \cite{GHK12}.
\end{defn}

\subsection{Admissible groups and Weyl groups.}

Given a Looijenga pair $(Y,D)$, the cone 
\[
\left\{ x\in{\rm Pic}\left(Y\right)_{\mathbb{R}}\mid x^{2}>0\right\} 
\]
has two connected components, let $\mathcal{C}^{+}$ be the component
containing all ample classes. Given an ample class $H$ in ${\rm Pic}(Y)$,
let 
\[
\tilde{\mathcal{M}}:=\left\{ E\in{\rm Pic}\left(Y\right)\mid E^{2}=K_{Y}\cdot E=-1\text{ and }E\cdot H>0\right\} 
\]
By Lemma 2.13 of \cite{GHK12}, $\tilde{\mathcal{M}}$ is independent
of the ample class we choose. Let $\mathcal{C}^{++}$ be the subcone
in $\mathcal{C}^{+}$ defined by the inequalities $x\cdot E\geq0$
for all $E$ in $\tilde{\mathcal{M}}$.
\begin{defn}
Given a Looijenga pair $(Y,D)$, we define the \emph{admissible group}
of $(Y,D)$, denoted by ${\rm Adm}_{(Y,D)}$, to be the subgroup of
the automorphism group of the lattice ${\rm Pic}(Y)$ preserving the
boundary classes $[D_{i}]$ and the cone $\mathcal{C}^{++}$.
\begin{defn}
Let $(Y,D)$ be a Looijenga pair.
\end{defn}

(1) An \emph{internal $(-2)$-curve }of $(Y,D)$ is a smooth rational
curve on $Y$of self-intersection $(-2)$ and disjoint from $D$.

(2) We say $(Y,D)$ is \emph{generic} if it has no internal $(-2)$-curves.
\begin{defn}
Let $(Y,D)$ be a Looijenga pair.

(1) We say a class $\left[\alpha\right]$ in ${\rm Pic}(Y)$ is a
\emph{root} if there is a family $\left(\mathcal{Y},\mathcal{D}\right)/S$,
a path $\gamma:[0,1]\rightarrow S$ and identifications
\[
\left(Y,D\right)\simeq\left(\mathcal{Y}_{\gamma(0)},\mathcal{D}_{\gamma(0)}\right),\quad\left(Y^{'},D^{'}\right)\simeq\left(\mathcal{Y}_{\gamma(1)},\mathcal{D}_{\gamma(1)}\right)
\]
 such that the isomorphism 
\[
H^{2}\left(Y,\mathbb{Z}\right)\simeq H^{2}\left(Y^{'},\mathbb{Z}\right)
\]
induced by the parallel transport along $\gamma$ sends $[\alpha]$
to an internal $(-2)$-curve $C$ on $Y^{'}$.

(2) Denote the set of roots $(Y,D)$ by $\Phi$. Given $\alpha$ in
${\rm \Phi},$ we can define the following reflection automorphism
$r_{\alpha}$ of ${\rm Pic}(Y)$:
\begin{align*}
r_{\alpha}:{\rm Pic}\left(Y\right)\rightarrow & {\rm Pic}\left(Y\right)\\
\beta\mapsto & \beta+\left\langle \alpha,\beta\right\rangle \alpha.
\end{align*}
We define the \emph{Weyl group }$W(Y,D)$ of $(Y,D)$ to be the subgroup
of ${\rm Aut}\left({\rm Pic}\left(Y\right)\right)$ generated by reflections
with respect to roots in $\Phi$.
\end{defn}

\end{defn}

By Theorem 5.1 of \cite{GHK12}, given a Looijenga pair $(Y,D)$,
the Weyl group $W(Y,D)$ is a normal subgroup of the admissible group
${\rm Adm}(Y,D)$.

\subsection{\label{subsec:Relationship-to-cluster}The relationship to rank $2$
cluster varieties.}

Let $(Y,D)$ be a generic Looijenga pair together with a toric model
$\rho:(Y,D)\rightarrow(\overline{Y},\overline{D})$. Let $\{E_{i}\}_{i=1}^{n}$
be the exceptional configuration corresponding to the given toric
model. Let $\overline{\Sigma}$ be the fan of the toric surface $\overline{Y}$
in a lattice $\Lambda\simeq\mathbb{Z}^{2}$. Set $N=\mathbb{Z}^{n}$
with basis $\{e_{i}\}_{i\in I}$ where $I:=\{1,\cdots,n\}$. Let $M$
be the dual lattice of $N$. For each $i\in I$, denote by $\overline{D}_{i}\subset D$
the irreducible component containing the center of the non-toric blowup
with exceptional locus $E_{i}$. Each $\overline{D}_{i}$ corresponds
to a ray in $\overline{\Sigma}$ with the primitive generator $w_{i}$
in $\Lambda$. If we assume that $w_{1},\cdots,w_{n}$ generate the
lattice $\Lambda$, then the map 
\begin{align*}
\varphi:N & \rightarrow\Lambda\\
e_{i} & \mapsto w_{i}
\end{align*}
 is surjective. Moreover, if we fix an isomorphism $\bigwedge^{2}\Lambda\simeq\mathbb{Z}$,
then the wedge product on $\Lambda$ pull-back to an integral valued
skew-symmetric form $\{,\}$ on $N$:
\[
\{n_{1},n_{2}\}=\varphi(n_{1})\wedge\varphi(n_{2}).
\]
Thus, from the fixed data $\left(N,\{,\},I\right)$, we define a seed
${\bf s}=\{e_{i}\}_{i\in I}$. Observe that the kernel $K$ of the
skew-symmetric form on $N$ has rank equal to ${\rm rank}N-2$, or
equivalently, the exchange matrix of ${\bf s}$ has rank $2$. Denote
by $\mathcal{X}_{{\bf s}}$ the corresponding $\mathcal{X}$-cluster
variety constructed from the seed ${\bf s}$. By section 5 of \cite{GHK15},
$\mathcal{X}_{{\bf s}}$ fibers over $T_{K^{*}}$ with generic fiber
agreeing up to codimension $2$ with a log Calabi-Yau surface deformation
equivalent to $Y\setminus D$. More precisely, let $\Sigma$ be the
fan in $M_{\mathbb{R}}$ consisting of rays generated by $-\{e_{i},\cdot\}$.
Denote by ${\rm TV}\left(\Sigma\right)$ the toric variety associated
to $\Sigma$. The projection $M\rightarrow M/K^{\perp}\simeq K^{*}$
induces a map 
\[
\overline{\lambda}:{\rm TV}\left(\Sigma\right)\rightarrow T_{K^{*}}
\]
For each $i$, let $D_{-\{e_{i},\cdot\}}$ be the toric divisor corresponding
to the ray generated by $-\{e_{i},\cdot\}$. Define the closed subscheme
\[
Z_{i}:=D_{-\{e_{i},\cdot\}}\bigcap\overline{V}\left(1+z^{e_{i}}\right)
\]
Let $\mathcal{Y}_{{\bf s}}$ be the blowup of the toric variety ${\rm TV}(\Sigma)$
along $\bigcup_{i=1}^{n}Z_{i}$ and $\mathcal{D}_{{\bf s}}$ the toric
boundary of $\mathcal{Y}_{{\bf s}}$. Let $\lambda:\mathcal{Y}_{{\bf s}}\rightarrow T_{K^{*}}$
be the fibration after the blowup. Then summarize results in section
$5$ of \cite{GHK15}, we have the following theorem:
\begin{thm}
\label{thm:rank 2 clus. var. and log CY surf.}For $\phi\in T_{K^{*}}$
general, $\left(\mathcal{Y}_{{\bf s},\phi},\mathcal{D}_{{\bf s},\phi}\right)$
is a generic Looijenga pair deformation equivalent to $(Y,D)$ with
the period point $\phi$. Moreover, $\mathcal{Y}_{{\bf s},\phi}\setminus\mathcal{D}_{{\bf s},\phi}$
agrees with $\mathcal{X}_{{\bf s},\phi}$ away from codimension $2$.
\end{thm}

By Theorem \ref{thm:rank 2 clus. var. and log CY surf.} we could
view a generic fiber in $\mathcal{X}_{{\bf s}}\rightarrow T_{K^{*}}$,
away from codimension $2$, as a log Calabi-Yau surface deformation
equivalent to $U=Y\setminus D$. Moreover, as proved in Theorem 5.5
of \cite{GHK15}, there is natural isomorphism $K\simeq D^{\perp}$
and we can view $\mathcal{X}_{{\bf s}}$ as the universal family of
$U=Y\setminus D$. Hence the title of this paper is justified.

For more detailed description of the relationship between Looijenga
pairs and rank $2$ $\mathcal{X}$-cluster varieties, see section
5 of \cite{GHK15} and section 2 of \cite{TM}.

\subsection{\label{subsec:admissible-cluster modular}Admissible groups and cluster
modular groups.}

Now let ${\bf s}=\left(e_{i}\right)_{i\in I}$ be an initial seed
constructed from a toric model $\rho:(Y,D)\rightarrow(\overline{Y},\overline{D})$
for a generic Looijenga pair $(Y,D)$. Let $\{E_{i}\}_{i\in I}$ be
the exceptional configuration of $\rho$. Let $v_{0}$ be the root
of $\mathfrak{T}_{{\bf s}}$. Given any $v$ in $\mathfrak{T}_{{\bf s}}$,
let $\mathcal{Y}_{{\bf s}_{v}}$ be the blowup of toric variety ${\rm TV}(\Sigma_{{\bf s}_{v}})$
as described in subsection \ref{subsec:Relationship-to-cluster}.
By section 5 of \cite{GHK15}, the birational map 
\[
\mu_{v_{0},v}:T_{M,{\bf s}}\dashrightarrow T_{M,{\bf s}_{v}}
\]
extends to a birational map between blowups of toric varieties, 
\[
\mu_{v_{0},v}:\mathcal{Y}_{{\bf s}}\dashrightarrow\mathcal{Y}_{{\bf s}_{v}}
\]
such that for $\phi\in T_{K^{*}}$ general, $\mu_{v_{0},v}$ restrict
to an biregular isomorphism $\mathcal{Y}_{{\bf s},\phi}\setminus\mathcal{D}_{{\bf s},\phi}\rightarrow\mathcal{Y}_{{\bf s}_{v},\phi}\setminus\mathcal{D}_{{\bf s}_{v},\phi}$.
Let $\{E_{i,\phi}\}_{i\in I}$ be the exceptional configuration of
\[
\rho_{\phi}:\left(\mathcal{Y}_{{\bf s},\phi},\mathcal{D}_{{\bf s},\phi}\right)\rightarrow(\overline{Y}_{{\bf s}},\overline{D}_{{\bf s}})
\]
Now suppose ${\bf s}_{v}$ is isomorphic to ${\bf s}$ and $\mu_{v_{0},v}^{t}:M_{\mathbb{R}}\rightarrow M_{\mathbb{R}}$\footnote{This is the geometric tropicalization using the min-plus convention.
Fock-Goncharov tropicalization in \cite{GHKK18} uses the max-plus
convention. For more details on this two tropicalizations, see section
2 of \cite{GHKK18}.} restricts to an isomorphism between fans $\Sigma_{{\bf s}}\rightarrow\Sigma_{{\bf s}_{v}}$.
Then the biregular isomorphism 
\[
\mu_{v_{0},v}:\mathcal{Y}_{{\bf s},\phi}\setminus\mathcal{D}_{{\bf s},\phi}\rightarrow\mathcal{Y}_{{\bf s}_{v},\phi}\setminus\mathcal{D}_{{\bf s}_{v},\phi}
\]
extends to a biregular isomorphism$\mu_{v_{0},v}:\mathcal{Y}_{{\bf s},\phi}\rightarrow\mathcal{Y}_{{\bf s}_{v},\phi}$.
Let $\gamma:{\bf s}\rightarrow\mathbf{s}_{v}$ be a seed isomorphism
compatible with the isomorphism between fans $\mu_{v_{0},v}^{t}:\Sigma_{{\bf s}}\rightarrow\Sigma_{{\bf s}_{v}}$,
i.e., for any $i\in I$, 
\[
\mu_{v_{0},v}^{t}(-\{e_{i},\cdot\})=\{-\gamma(e_{i}),\cdot\}
\]
The isomorphism 
\[
\xyR{2pc}\xyC{3pc}\xymatrix{T_{M,{\bf s}_{v}}\ar[r]^{\gamma^{-1}}\ar[d] & T_{M,{\bf s}}\ar[d]\\
T_{K^{*}}\ar[r]^{\gamma^{-1}} & T_{K^{*}}
}
\]
induced by $\gamma^{-1}$ extends to an isomorphism
\[
\xyR{2pc}\xyC{3pc}\xymatrix{\mathcal{Y}_{{\bf s}_{v}}\ar[r]^{\gamma^{-1}}\ar[d] & \mathcal{Y}_{{\bf s}}\ar[d]\\
T_{K^{*}}\ar[r]^{\gamma^{-1}} & T_{K^{*}}
}
\]
Now consider the isomorphism 
\[
\gamma^{-1}\circ\mu_{v_{0},v}:\left(\mathcal{Y}_{{\bf s},\phi},\mathcal{D}_{{\bf s},\phi}\right)\rightarrow\left(\mathcal{Y}_{{\bf s},\gamma^{-1}\phi},\mathcal{D}_{{\bf s},\gamma^{-1}\phi}\right)
\]
Let $\left\{ E_{i,\gamma^{-1}\phi}\right\} _{i\in I}$ be the exceptional
configuration of 
\[
\rho_{\gamma^{-1}\phi:}\left(\mathcal{Y}_{{\bf s},\gamma^{-1}\phi},\mathcal{D}_{{\bf s},\gamma^{-1}\phi}\right)\rightarrow\left(\overline{Y}_{{\bf s}},\overline{D}_{{\bf s}}\right)
\]
Since $\gamma^{-1}\circ\mu_{v_{0},v}$ induces trivial map on the
boundary, 
\[
\left\{ E_{i,\phi}^{'}\right\} _{i\in I}:=\left\{ \left(\gamma^{-1}\circ\mu_{v_{0},v}\right)^{*}\left(E_{i,\gamma^{-1}\phi}\right)\right\} _{i\in I}
\]
is another exceptional configuration of $\rho_{\phi}$ and there exists
an unique element $\delta$ in the admissible group ${\rm Adm}(Y,D)$
that maps $\{E_{i,\phi}\}_{i\in I}$ to $\{E_{i,\phi}^{'}\}_{i\in I}$.
Thus, we can identify $\delta$ with $\gamma^{-1}\circ\mu_{v_{0},v}$
as an element in the cluster modular group.

\subsection{Classification of positive Looijenga pairs.}

The following proposition follows from Theorem 4.4 and Theorem 4.5
of \cite{TM}:
\begin{prop}
\label{prop:classification rank 2 clus. var.}Let $(Y,D)$ be a generic
positive Looijenga pair. Then $\mathcal{X}_{(Y,D)}$ is non-acyclic,
that is, it cannot be constructed from an acyclic seed if and only
if
\[
D^{\perp}\simeq{\bf D}_{n}\,(n\geq4)\text{ or }{\bf E}_{n}\,(n=6,7,8).
\]
\end{prop}

Since the theory of acyclic cluster varieties are well-understood,
for the rest of this paper, we will focus only on the cases where
$D^{\perp}\simeq{\bf D}_{n}\,(n\geq4)\text{ or }{\bf E}_{n}\,(n=6,7,8)$.

\section{The geometry of folding\label{sec:The-geometry-of}}

In this section, we first briefly review basics of scattering diagrams
that arise in cluster theory. Readers familiar with the Gross-Siebert
program could skip this part. Then we recall the folding technique
developed in \cite{YZ} in the context of cluster scattering diagrams.
In the last subsection \ref{subsec:Folding-and-degenerate}, we review
the geometric meaning of the folding procedure for $\mathcal{X}$-cluster
varieties associated to log Calabi-Yau surfaces.

\subsection{Cluster scattering diagrams.\label{subsec:Cluster-scattering-diagrams.}}

Let $N$ be a lattice of rank $n$ with dual lattice $M$ and a skew-symmetric
$\mathbb{Z}$-bilinear form $\{,\}:N\times N\rightarrow\mathbb{\mathbb{Z}}.$
Let $I=\{1,2,\cdots,n\}$ be the index set. Fix an initial seed ${\bf s}$,
i.e., ${\bf s}=(e_{i})_{i=1}^{n}$ is a labelled basis of $N$. For
simplicity, we assume that we do not have frozen variables for the
entire paper. In general when we fix the lattice $N$, we also fix
a sublattice $N^{\circ}\subset N$ and positive integers $d_{i}$
for $i\in I$ such that $\{d_{i}e_{i}\mid i\in I\}$ is a basis for
$N^{\circ}$. As in \cite{YZ}, we assume that the original seed ${\bf s}$
is simply-laced, that is, $d_{i}=1$ for all $i\in I$. 

Define 
\[
N^{+}=\{\sum_{i\in I}a_{i}e_{i}\mid a_{i}\in\mathbb{N},\sum_{i\in I}a_{i}>0\}
\]
and the following algebraic torus over $\mathbb{C}$:
\[
T_{M}=M\otimes_{\mathbb{Z}}\mathbb{G}_{m}={\rm Hom}(N,\mathbb{G}_{m})={\rm Spec}(\mathbb{C}[N]).
\]
The skew-symmetric form $\{,\}$ on $N$ gives the algebraic torus
$T_{M}$ a Poisson structure:
\begin{equation}
\{z^{n},z^{n^{'}}\}=\{n,n^{'}\}z^{n+n^{'}}.\label{eq:poisson structure}
\end{equation}
Let $\mathfrak{g}=\bigoplus_{n\in N^{+}}\mathfrak{g}_{n}$ where $\mathfrak{g}_{n}=\mathbb{C}\cdot z^{n}$.
The Poisson bracket $\{,\}$ on $\mathbb{C}[N]$ endows $\mathfrak{g}$
with a Lie bracket $[,]$: $[f,g]=-\{f,g\}$. By \ref{eq:poisson structure},
it is clear that $\mathfrak{g}$ is graded by $N^{+}$, i.e., $[\mathfrak{g}_{n_{1}},\mathfrak{g}_{n_{2}}]\subset\mathfrak{g}_{n_{1}+n_{2}}$
and $\mathfrak{g}$ is skew-symmetric for $\{,\}$ on $N$, i.e.,
$[\mathfrak{g}_{n_{1}},\mathfrak{g}_{n_{2}}]=0$ if $\{n_{1},n_{2}\}=0$.

Define the following linear functional on $N$:
\begin{align}
d:N & \rightarrow\mathbb{Z}\nonumber \\
\sum_{i\in I}a_{i}e_{i} & \mapsto\sum_{i\in I}a_{i}.\label{eq:-11}
\end{align}
For each $k\in\mathbb{Z}_{\geq0}$, there is an Lie subalgebra $\mathfrak{g}^{>k}:=\bigoplus_{n\in N^{+}\mid d(n)>k}\mathfrak{g}_{n}$
of $\mathfrak{g}$ and a $N^{+}$-graded nilpotent Lie algebra $\mathfrak{g}^{\leq k}:=\mathfrak{g}/\mathfrak{g}^{>k}.$
We denote the unipotent Lie group corresponding to $\mathfrak{g}^{\leq k}$
as $G^{\leq k}$. As a set, it is in bijection with $\mathfrak{g}^{\leq k}$,
but the group multiplication is given by the Baker-Campbell-Hausdorff
formula. Given $i<j$, we have canonical homomorphism:
\[
\mathfrak{g}^{\leq j}\rightarrow\mathfrak{g}^{\leq i},\quad G^{\leq j}\rightarrow G^{\leq i}.
\]
Define the pro-nilpotent Lie algebra and its corresponding pro-unipotent
Lie group by taking the projective limits:
\[
\hat{\mathfrak{g}}=\lim_{\longleftarrow}\mathfrak{g}^{\leq k},\quad G=\lim_{\longleftarrow}G^{\leq k}.
\]
We have canonical bijections:
\[
\exp:\mathfrak{g}^{\leq k}\rightarrow G^{\leq k},\quad\exp:\hat{\mathfrak{g}}\rightarrow G.
\]
Given an element $n_{0}$ in $N^{+}$, there is an Lie subalgebra
$\mathfrak{g}_{n_{0}}^{\parallel}:=\bigoplus_{k\in\mathbb{Z}_{>0}}\mathfrak{g}_{k\cdot n_{0}}$
of $\mathfrak{g}$ and its corresponding Lie subgroup $G_{n_{0}}^{\parallel}:=\exp(\widehat{\mathfrak{g}_{n_{0}}^{\parallel}})$
of $G$. By our assumption that $\mathfrak{g}$ is skew-symmetric
for $\{,\}$, $\mathfrak{g}_{n_{0}}^{\parallel}$ and therefore $G_{n_{0}}^{\parallel}$
are abelian. 
\begin{defn}
A \emph{wall} in $M_{\mathbb{R}}$ for $N^{+}$ and $\mathfrak{g}$
is pair $(\mathfrak{d},g_{\mathfrak{d}})$ such that 

1) $g_{\mathfrak{d}}$ is contained in $G_{n_{0}}^{\parallel}$ for
some primitive element $n_{0}$ in $N^{+}$

2) $\mathfrak{d}$ is contained in $n_{0}^{\perp}\subset M_{\mathbb{R}}$
and is a convex rational polyhedral cone of codimension $1$. 
\end{defn}

\begin{defn}
A \emph{scattering diagram} $\mathfrak{D}$ for $N^{+}$ and $\mathfrak{g}$
is a collection of walls such that for every order $k$ in $\mathbb{Z}_{\geq0}$,
there are only finitely many walls $(\mathfrak{d},g_{\mathfrak{d}})$
in $\mathfrak{D}$ with non-trivial image of $g_{\mathfrak{d}}$ under
the projection map $G\rightarrow G^{\leq k}$. 

Given a wall $(\mathfrak{d},g_{\mathfrak{d}})$ such that $\mathfrak{d}$
is contained in $n_{0}^{\perp}$ for a primitive element $n_{0}\in N^{+}$,
we say that $(\mathfrak{d},g_{\mathfrak{d}})$ is \emph{incoming}
if $\{n_{0},\cdot\}$ is contained in $\mathfrak{d}$. Otherwise,
we say that $(\mathfrak{d},g_{\mathfrak{d}})$ is \emph{outgoing}.
\end{defn}

\begin{thm}
\label{thm: uniqueness-incoming walls}(Theorem 1.12 and Theorem 1.21
of \cite{GHKK18}) The equivalence class of a consistent scattering
diagram is determined by its set of incoming walls. Vice versa, let
$\mathfrak{D}_{{\rm in}}$ be a scattering diagram whose only walls
are full hyperplanes, i.e., are of the form $\left(n_{0}^{\perp},g_{n_{0}}\right)$
where $n_{0}$ is a primitive element in $N^{+}$. Then there is a
scattering diagram $\mathfrak{D}$ satisfying:

(1) $\mathfrak{D}$ is consistent,

(2) $\mathfrak{D}\supset\mathfrak{D}_{{\rm in}}$,

(3) $\mathfrak{D}\setminus\mathfrak{D}_{{\rm in}}$ consists only
of outgoing walls.

Moreover, $\mathfrak{D}$ satisfying these three properties is unique
up to equivalence.
\end{thm}

Define the following set of initial walls associated to the seed ${\bf s}=(e_{i})_{i\in I}$:
\[
\mathfrak{D}_{{\bf s},{\rm in}}=\left\{ \left(e_{i}^{\perp},\exp\left(-{\rm Li}_{2}(-z^{e_{i}})\right)\right)\mid i\in I\right\} 
\]
where ${\rm Li}_{2}$ is the dilogarithm function ${\rm Li}_{2}(x)=\sum_{k\geq1}\frac{x^{k}}{k^{2}}$.
By Theorem \ref{thm: uniqueness-incoming walls}, there is a consistent
scattering diagram $\mathfrak{D}_{{\bf s}}$ , unique up to equivalence,
whose set of initial walls is $\mathfrak{D}_{{\rm {\bf s},{\rm in}}}$.
It is the \emph{cluster scattering diagram} associated to the seed
${\bf s}$. Slightly different from \cite{GHKK18}, we consider a
single scattering diagram with wall-crossing automorphisms in the
Lie group. By letting the Lie group acting on different completed
rings, we get wall-crossings for both $\mathcal{X}$ and $\mathcal{A}_{{\rm prin}}$
spaces.

Define 
\[
\mathcal{C}_{{\bf s}}^{+}=\{m\in M_{\mathbb{R}}\mid m\mid_{N^{+}}\geq0\},\quad\mathcal{C}_{{\bf s}}^{-}=\{m\in M_{\mathbb{R}}\mid m\mid_{N^{+}}\leq0\}.
\]
By Theorem 2.13 of \cite{GHKK18}, $\mathcal{C}_{{\bf s}}^{+}$ and
$\mathcal{C}_{{\bf s}}^{-}$ are two distinguished chambers in $\mathfrak{D}_{{\bf s}}$.
Moreover, $\mathcal{C}_{{\bf s}}^{+}$ (resp. $\mathcal{C}_{{\bf s}}^{-}$)
is a maximal cone of a simplicial fan in $\mathfrak{D}_{{\bf s}}$
which we denote by $\Delta_{{\bf s}}^{+}$ (resp. $\Delta_{{\bf s}}^{-}$).
We call $\Delta_{{\bf s}}^{+}$ the\emph{ cluster complex} of $\mathfrak{D}_{{\bf s}}$.
For more details on the structure of the cluster complex and its description
via tropicalization, see section 2 of \cite{GHKK18}. 

\subsection{\label{subsec:Summary-folding}Summary of the folding procedure.}

Suppose $\Pi$ is a subgroup of the symmetric group $S_{n}$ acting
on $I$ satisfies the following condition:

For any indices $i,j\in I$, for any $\pi_{1},\pi_{2}$ in $\Pi$,
\begin{equation}
\{e_{i},e_{j}\}=\{e_{\pi_{1}\cdot i},e_{\pi_{2}\cdot j}\}.\label{eq: symmetry of seed}
\end{equation}
Then we can apply the \emph{folding} procedure to construct a new
seed $\overline{{\bf s}}$. For any $i\in I$, denote by $\Pi i$
the orbit of $i$ under the action of $\Pi$. Let $\overline{I}$
be the set indexing the orbits of $I$ under the action of $\Pi$.
Let $\overline{N}$ be the lattice with basis $\{e_{\Pi i}\}_{\Pi i\in\overline{I}}$.
Define a $\mathbb{Z}$-valued skew-symmetric bilinear form $\{,\}$
on $\overline{N}$ such that $\{e_{\Pi i},e_{\Pi j}\}=\{e_{i},e_{j}\}$.
By condition \ref{eq: symmetry of seed}, the skew-symmetric form
on $\overline{N}$ is well-defined. There is a natural quotient homomorphism
of lattices $q:N\rightarrow\overline{N}$ that sends $e_{i}$ to $e_{\Pi i}$.
Now we have the following fixed data:
\begin{itemize}
\item The lattice $\overline{N}$ with the $\mathbb{Z}$-valued skew-symmetric
bilinear form $\{,\}$. 
\item The index set $\overline{I}$ parametrizing the set of orbits of $I$
under the action of $\Pi$.
\item For each $\Pi i\in\overline{I}$, a positive integer $d_{\Pi i}=|q^{-1}\{e_{\Pi i}\}|$. 
\end{itemize}
We define the following new seed: 
\[
\bar{{\bf s}}=(e_{\Pi i}\mid\Pi i\in\bar{I}).
\]
Let $\overline{N}^{\circ}\subset\overline{N}$ be the sublattice such
that $\{d_{\Pi i}e_{\Pi i}\mid i\in\overline{I}\}$ is a basis. Unlike
\cite{GHKK18}, we do not impose the condition that the greatest common
divisor of $\{d_{\Pi i}\}_{\Pi i\in\overline{I}}$ need to be $1$.
Notice that the new seed $\overline{{\bf s}}$ may no longer be simply
laced since $d_{\Pi i}$ may not be $1$. The seed $\overline{{\bf s}}$
defines the exchange matrix $\overline{\epsilon}$ such that $\overline{\epsilon}_{\Pi i\,\Pi j}=\{e_{\Pi i},e_{\Pi j}\}d_{\Pi j}.$
Take 
\[
\overline{N}^{+}=\{\sum_{\Pi i\in\bar{I}}a_{\Pi i}e_{\Pi i}\mid a_{\Pi i}\geq0,\sum a_{\Pi i}>0\}.
\]
Let $\mathfrak{\bar{g}}\subset\mathbb{C}[\overline{N}]$ be the $\mathbb{C}$-vector
space with basis $z^{\bar{n}},\bar{n}\in\overline{N}^{+}$. Define
a Lie bracket on $\bar{\mathfrak{g}}$ as follows:
\[
[z^{e_{\Pi i}},z^{e_{\Pi j}}]=-\{e_{\Pi i},e_{\Pi j}\}z^{e_{\Pi i}+e_{\Pi j}}.
\]
Define $\tilde{q}:\mathfrak{g}\rightarrow\overline{\mathfrak{g}}$
by $z^{n}\mapsto z^{\overline{n}}$ where $\bar{n}$ is the image
of $n$ under the surjective quotient homomorphism $q:N\rightarrow\overline{N}$.
Since
\begin{align*}
[\tilde{q}\left(z^{n}\right),\tilde{q}(z^{n^{'}})] & =-\{\overline{n},\overline{n^{'}}\}z^{\overline{n+n^{'}}}=-\{n,n^{'}\}z^{\overline{n+n^{'}}}=\tilde{q}\left([z^{n},z^{n^{'}}]\right),
\end{align*}
$\tilde{q}$ is a surjective graded Lie algebra homomorphism. 

As in the case of $\mathfrak{g}$, we can complete $\bar{\mathfrak{g}}$
with respect to its grading. Let $\overline{G}$ be the corresponding
pro-nilpotent group after the completion. Then $\tilde{q}$ induces
a graded Lie group homomorphism from $G$ to $\overline{G}$. 

By Theorem \ref{thm: uniqueness-incoming walls}, the seed $\overline{{\bf s}}$
gives rises to a consistent scattering diagram $\mathfrak{D}_{\overline{{\bf s}}}$
whose set of incoming walls are defined as follows: 
\[
\mathfrak{D}_{\overline{{\bf s}},{\rm in}}=\left\{ \left(e_{\Pi i}^{\perp},\exp\left(-d_{\Pi i}{\rm Li}_{2}(-z^{e_{\Pi i}})\right)\right)\mid\Pi i\in I\right\} .
\]
Denote by $\overline{M}$ the dual basis of $\overline{N}$ and $\overline{N}^{\circ}$
the sublattice of $\overline{N}$ with basis $\{d_{\Pi i}e_{\Pi i}\}_{\Pi i\in\overline{I}}$.
Let $\overline{M}^{\circ}$ be the dual lattice of $\overline{N}^{\circ}$.
Let 
\[
q^{*}:\overline{M}_{\mathbb{R}}\hookrightarrow M_{\mathbb{R}}
\]
the embedding induced by the quotient homomorphism $q:N\rightarrow\overline{N}$.
One of the main technical result in \cite{YZ} is that the quotient
construction of folding extends to the level of scattering diagrams:
\begin{thm}
(cf. Construction 2.19, Theorem 2.20, and Corollary 2.21 of \cite{YZ})
\label{thm:folding-scattering diag}The scattering diagram $\mathfrak{D}_{\overline{{\bf s}}}$
can be obtained from $\mathfrak{D}_{{\bf s}}$ via a quotient construction. 
\end{thm}

For more details of the folding construction in the context of scattering
diagrams and cluster theory and its application to the Fock-Goncharov
canonical basis conjecture, see section 2 and 3 of \cite{YZ}.

For the rest of this subsection, we prove some lemmas that will be
used later. Let $\mathfrak{T}_{{\bf s}}$ (resp. $\mathfrak{T}_{\overline{{\bf s}}}$)
the oriented tree parametrizing all seeds mutationally equivalent
to ${\bf s}$ (resp. $\overline{{\bf s}}$). Given $k$ in $I$, let
$l_{1},\cdots,l_{|\Pi k|}$ be an enumeration of elements in $|\Pi k|$.
\begin{lem}
\label{lem:Pi-constrained mutation}Let ${\bf s}^{'}=(e_{i}^{'})_{i\in I}=\mu_{l_{1}}\circ\cdots\circ\mu_{l_{|\Pi k|}}({\bf s})$.
Then ${\bf s}^{'}$ does not depend on the choice of ordering of elements
in $\Pi k$. With respect to the new seed ${\bf s}^{'}$, the action
of $\Pi$ still satisfies the condition for folding, that is, for
any $\pi_{1},\pi_{2}$ in $\Pi$, 
\[
\{e_{\pi_{1}\cdot i}^{'},e_{\pi_{2}\cdot j}^{'}\}=\{e_{i}^{'},e_{j}^{'}\}.
\]
Moreover, $\mu_{\Pi k}(\overline{{\bf s}})=(q(e_{i}^{'}))_{\Pi i\in\overline{I}}$
where $q:N\twoheadrightarrow\overline{N}$ is the quotient map of
lattice sending $e_{i}$ to $e_{\Pi i}$. 
\end{lem}

\begin{proof}
First observe that for any $1\leq a\leq|\Pi k|$, $e_{l_{a}}^{'}=-e_{l_{a}}$.
Therefore for any $1\leq a,b\leq|\Pi k|$, $\{e_{l_{a}}^{'},e_{l_{b}}^{'}\}=0$.
If $i,j\in I$ are not in $\Pi k$, then for any $\pi_{1},\pi_{2}$
in $\Pi$,
\[
\begin{split}e_{\pi_{\epsilon}\cdot i}^{'} & =e_{\pi_{1}\cdot i}+\sum_{a=1}^{|\Pi k|}[\{e_{\pi_{\epsilon}\cdot i},e_{l_{a}}\}]_{+}e_{l_{a}}=e_{\pi_{\epsilon}\cdot i}+[\{e_{i},e_{k}\}]_{+}\sum_{a=1}^{|\Pi k|}e_{l_{a}},\quad\epsilon=1,2.\end{split}
\]
Therefore, 
\begin{align*}
\{e_{\pi_{1}\cdot i}^{'},e_{\pi_{2}\cdot j}^{'}\} & =\{e_{\pi_{1}\cdot i}+[\{e_{i},e_{k}\}]_{+}\sum_{a=1}^{|\Pi k|}e_{l_{a}},\,e_{\pi_{2}\cdot j}+[\{e_{j},e_{k}\}]_{+}\sum_{a=1}^{|\Pi k|}e_{l_{a}}\}=\{e_{i}^{'},e_{j}^{'}\}.
\end{align*}
If $i\in I$ is not in $\Pi k$, then for any $1\leq a\leq|\Pi k|$,
\begin{align*}
\{e_{\pi_{1}\cdot i}^{'},e_{l_{a}}^{'}\} & =\{e_{\pi_{1}\cdot i}+[\{e_{i},e_{k}\}]_{+}\sum_{a=1}^{|\Pi k|}e_{l_{a}},-e_{l_{a}}\}=\{e_{i},-e_{k}\}=\{e_{i}^{'},e_{k}^{'}\}.
\end{align*}
Hence the action of $\Pi$ on $I$ with respect to the new seed ${\bf s}^{'}$
still satisfies the folding condition. It remains to prove the second
part of the lemma. Notice that for any $1\leq a\leq|\Pi k|$, $q(e_{l_{a}}^{'})=-e_{\Pi k}=\mu_{\Pi k}(e_{\Pi k})$.
If $i\in I$ is not in $\Pi k$, then for any $\pi$ in $\Pi$, 
\begin{align*}
q(e_{\pi\cdot i}^{'}) & =q\left[e_{\pi\cdot i}+[\{e_{i},e_{k}\}]_{+}\sum_{a=1}^{|\Pi k|}e_{l_{a}}\right]=\mu_{\Pi k}(e_{\Pi i}^{'}).
\end{align*}
\end{proof}
\begin{defn}
Denote by $[{\bf s}]$ the set of seeds mutationally equivalent to
${\bf s}$. Given $k$ in $I$, by Lemma \ref{lem:Pi-constrained mutation},
there is a well-defined composition of mutations 
\begin{equation}
\mu_{\Pi k}=\prod_{k^{'}\in\Pi k}\mu_{k^{'}}.\label{eq:-1}
\end{equation}
We call such compositions of mutations $\mu_{\Pi k}$ a \emph{$\Pi$-constrained
mutation}. Moreover, also by Lemma \ref{lem:Pi-constrained mutation},
the new seed $\mu_{\Pi k}({\bf s})$ is an element in $[{\bf s}]$
and $\Pi$-constrained mutation is also well-defined for $\mu_{\Pi k}({\bf s})$.
We denote by $[{\bf s}]^{\Pi}\subset[{\bf s}]$ the set of seeds related
to ${\bf s}$ by $\Pi$-constrained mutations.

Given ${\bf s}_{v}\in[{\bf s}]$, let $G^{v}$ be the $\mathbf{g}$-matrix
of the seed ${\bf s}_{v}$. We say $G^{v}$ is \emph{$\Pi$-invariant}
if for any $\pi$ in $\Pi$, for any $k\in I$, $\pi\cdot G_{k}^{v}=G_{\pi\cdot k}^{v}.$
We say the chamber $\mathcal{C}_{v}^{+}$ in $\Delta_{{\bf s}}^{+}$
is \emph{$\Pi$-invariant }if $G^{v}$ is. In particular, for any
$\pi$ in $\Pi$, $\pi(\mathcal{C}_{v}^{+})=\mathcal{C}_{v}^{+}$
under the action of $\pi$ on $\mathfrak{D}_{{\bf s}}$ (cf. section
2.2 of \cite{YZ}).
\end{defn}

\begin{lem}
\label{lem:Pi-invariant chambers}Let ${\bf s}^{'}=\mathbf{s}_{w}$
be a seed in $[{\bf s}]^{{\rm iso}}$. Suppose $\mathcal{C}_{w}^{+}\in\Delta_{{\bf s}}^{+}$
is $\Pi$-invariant. Then given any $v^{'}$ in $\mathfrak{T}_{{\bf s}^{'}}\subset\mathfrak{T}_{{\bf s}}$
such that ${\bf s}_{v^{'}}$ is contained in $[{\bf s}^{'}]^{\Pi}$,
$\mathcal{C}_{v^{'}}^{+}\in\Delta_{{\bf s}}^{+}$ is $\Pi$-invariant.
In particular, $\mathcal{C}_{v^{'}}^{+}\bigcap q^{*}(\overline{M}_{\mathbb{R}})$
is of full dimension in $q^{*}\left(\overline{M}_{\mathbb{R}}\right)$.
\end{lem}

\begin{proof}
It suffices to prove for the case where ${\bf s}_{v^{'}}=\mu_{\Pi k}\left({\bf s}_{w}\right)$.
Indeed, since ${\bf s}_{w}$ is isomorphic to ${\bf s}$, up to permutation
of indices, we could assume that ${\bf s}_{w}$ and $\Pi$ satisfy
the folding condition \ref{eq: symmetry of seed}. Let $G^{w}$ ,
$C^{w}$ be the ${\bf g}$-matrix and $\mathbf{c}$-matrix associated
to $\mathcal{C}_{w}^{+}$ respectively. By sign coherence of $\mathbf{c}$-vectors,
all entries of $C_{k}^{w}$ either $\geq0$ or $\leq0$.

Case 1: All entries of $C_{k}^{w}\geq0$. By the $\Pi$-invariance
of $\mathcal{C}_{w}^{+}$ , the same holds for any $k^{'}\in\Pi k$.
By the mutation formula for $\mathbf{g}$-vectors, for any $k^{'}\in\Pi k$,
\[
G_{k^{'}}^{v^{'}}=-G_{k^{'}}^{w}+\sum_{i=1}^{n}\left[-\epsilon_{ik^{'}}^{w}\right]_{+}G_{i}^{w}.
\]
Let $k^{'}=\pi\cdot k$. Then it follows from the $\Pi$-invariance
of $G^{w}$ that
\begin{align*}
\pi\cdot G_{k}^{v^{'}} & =-\pi\cdot G_{k}^{w}+\sum_{i=1}^{n}\left[-\epsilon_{ik}^{w}\right]_{+}\pi\cdot G_{i}^{w}=-G_{\pi\cdot k}^{w}+\sum_{i=1}^{n}\left[-\epsilon_{ik^{'}}^{w}\right]_{+}G_{\pi\cdot i}^{w}\\
 & =-G_{k^{'}}^{w}+\sum_{i=1}^{n}\left[-\epsilon_{ik^{'}}^{w}\right]_{+}G_{\pi\cdot i}^{w}.
\end{align*}
Since for any $i\in I$, 
\[
\epsilon_{ik^{'}}=\{e_{i},e_{k^{'}}\}=\{e_{\pi\cdot i},e_{k^{'}}\}=\epsilon_{(\pi\cdot i)k^{'}},
\]
we conclude that
\begin{align*}
\pi\cdot G_{k}^{v^{'}} & =-G_{k^{'}}^{w}+\sum_{i=1}^{n}\left[-\epsilon_{ik^{'}}^{w}\right]_{+}G_{\pi\cdot i}^{w}=-G_{k^{'}}^{w}+\sum_{i=1}^{n}\left[-\epsilon_{(\pi\cdot i)k^{'}}^{w}\right]_{+}G_{\pi\cdot i}^{w}\\
 & =-G_{k^{'}}^{w}+\sum_{i=1}^{n}\left[-\epsilon_{ik^{'}}^{w}\right]G_{i}^{w}=G_{k^{'}}^{v}.
\end{align*}
Therefore $G^{v^{'}}$ and hence $\mathcal{C}_{v^{'}}^{+}$ is $\Pi$-invariant.

Case 2: All entries of $C_{k}^{w}\leq0$. Then we have 
\[
G_{k^{'}}^{v^{'}}=-G_{k^{'}}^{w}+\sum_{i=1}^{n}\left[\epsilon_{ik^{'}}^{w}\right]_{+}G_{i}^{w}.
\]
Similar to Case 1, we again conclude that $\mathcal{C}_{v^{'}}^{+}$
is $\Pi$-invariant.
\end{proof}

\subsection{\label{subsec:Folding-and-degenerate}Folding and degenerate loci
of families of log Calabi-Yau surfaces.}

Let $(Y,D)$ be a generic Looijenga pair together with a toric model
$\rho:(Y,D)\rightarrow(\overline{Y},\overline{D})$. Then as described
in subsection \ref{subsec:Relationship-to-cluster}, $(Y,D)$ together
with the toric model gives rise to seed data $\left(N,\{\cdot,\cdot\},{\bf s}\right)$.
Notice that folding construction naturally applies in this situation.
Given $i,j\in I$, if $\{e_{i},\}=\{e_{j},\}$, then $E_{i}$ and
$E_{j}$ meet the same boundary component and we can naturally fold
$e_{i}$ and $e_{j}$ together. More generally, let $\Pi$ be a subgroup
of the symmetric group $S_{n}$ generated by involutions such that
for any $\pi$ in $\Pi$, 
\[
\{e_{i},\cdot\}=\{e_{\pi\cdot i},\cdot\}.
\]
Then the action of $\Pi$ satisfies the folding condition and gives
us a folded seed $\overline{{\bf s}}$. Moreover, we can identify
$\Pi$ as a subgroup of the Weyl group $(Y,D)$. Indeed, given $i\in I$
and $j\in\Pi i\setminus\{i\}$, under the identification $D^{\perp}\simeq K$
given in Theorem 5.5 in\cite{GHK15}, $e_{i}-e_{j}$ corresponds the
root $E_{i}-E_{j}$ . Let $\pi$ be the involution $(ij)$ in $\Pi$.
Then the induced action of $\pi$ on $K\simeq D^{\perp}$ agrees with
the reflection $r_{E_{i}-E_{j}}$ with respect to the root $E_{i}-E_{j}$
in $W(Y,D)$.

Let ${\bf s}$ be a seed obtained from a generic Looijenga pair $(Y,D)$
with toric model $\rho:(Y,D)\rightarrow(\overline{Y},\overline{D})$
and $\overline{{\bf s}}$ the seed folded from ${\bf s}$ via the
action of a group ${\bf \Pi}$ on the index set $I$. Notice that
folding procedure does not change the rank of the exchange matrix,
so the $\mathcal{X}$-cluster variety associated to $\overline{{\bf s}}$
is still a rank $2$ cluster variety. Moreover, $\mathcal{X}_{\overline{{\bf s}}}$
is a locally closed cluster subvariety of $\mathcal{X}_{{\bf s}}$
and corresponds to a degenerate subfamily of the universal family
of $U=Y\setminus D$. To see this, let $\overline{K}$ be the kernel
of the skew-symmetric form on $\overline{N}$. Let $\overline{\mathcal{Y}}_{\overline{{\bf s}}}$
be the blowup of the toric variety given by $\overline{{\bf s}}$.
Given a general $\overline{\phi}$ in $T_{\overline{K}^{*}}$, $\mathcal{X}_{\overline{{\bf s}},\phi}$
up to codimension 2 agrees with $\overline{\mathcal{Y}}_{\overline{{\bf s}},\overline{\phi}}$,
which is blowup of of $\overline{Y}$ at a collection of distinct
points $p_{\Pi1},\cdots p_{\Pi n}$ where $p_{\Pi i}\in\overline{D}_{i}$
has multiplicity $|\Pi i|$. For each $i$ such that $|\Pi i|>1$,
the weighted blowup at $p_{\Pi i}$ gives rise to a du Val singularity
of type $A_{|\Pi i|-1}$ in $\overline{\mathcal{Y}}_{\overline{{\bf s}},\overline{\phi}}$.
Therefore, we could view $\mathcal{X}_{\overline{{\bf s}}}$ as a
degenerate subfamily of the universal family of $U$ such that generic
members in this subfamily have singularities specified by the combinatorial
data of the folding from ${\bf s}$ to $\bar{{\bf s}}$. 
\begin{example}
Consider $(\mathbb{P}^{2},\overline{D})$ where $\overline{D}=\overline{D}_{1}+\overline{D}_{2}+\overline{D}_{3}$
is a triangle of lines in $\mathbb{P}^{2}$. On each $\overline{D}_{i}$,
we blow up at two distinct points in the smooth locus of $\overline{D}$
and obtain two exceptional curves $E_{2i-1}$ and $E_{2i}$. The new
pair $(Y,D)$ we get is a cubic surface whose anticanonical cycle
$D$ is the strict transform of $\overline{D}$ via the blowup $\rho:(Y,D)\rightarrow(\overline{Y},\overline{D})$.
Let $l$ be the class of the pull-back of a line in $\mathbb{P}^{2}$
via $\rho$. We have $D^{\perp}\simeq\mathbf{D}_{4}$ and the Weyl
group of $(Y,D)$ is the Weyl group $W(\mathbf{D}_{4})$ of ${\bf D}_{4}$.
Moreover, $e_{1}-e_{2},e_{3}-e_{4},e_{5}-e_{6},e_{1}+e_{2}+e_{3}$
form a basis of $K$ which under the identification $K\simeq D^{\perp}$
correspond to roots $E_{1}-E_{2},E_{3}-E_{4},E_{5}-E_{6},l-E_{1}-E_{2}-E_{3}$.
These roots form a simple root system in ${\bf D}_{4}$. Let $\Pi$
be the group generated by involutions $(12),(34),(56)$. Then $\Pi$
can be identified as the subgroup generated by reflections with respect
to roots $E_{1}-E_{2}$, $E_{3}-E_{4},$ $E_{5}-E_{6}$. Let ${\bf s}$
be the seed given by $(Y,D)$ and $\overline{{\bf s}}$ be the folded
seed given by the action of $\Pi$ on the index set. Then given a
very general $\overline{\phi}\in T_{\overline{K}^{*}}$, $\mathcal{X}_{\overline{{\bf s}},\overline{\phi}}$
up to codimension $2$ agrees with a singular cubic surface having
the following description. We blow up at a interior point $p_{i,\overline{\phi}}$
in $\overline{D}_{i}$ twice for each $i$. Since we assume that $\overline{\phi}$
is general, $p_{1,\overline{\phi}},p_{2,\overline{\phi}},p_{3,\overline{\phi}}$
are not collinear and therefore we could blow down the three $(-2)$-curves
simultaneously and get a singular cubic surface with three ordinary
double points.
\end{example}

\section{The action of Weyl groups on scattering diagrams: positive cases\label{sec:action-Weyl group-scattering diag}}

Throughout this section, we will focus on positive Looijenga pairs
$(Y,D)$ such that the corresponding $\mathcal{X}$-cluster variety
is non-acyclic. By Proposition \ref{prop:classification rank 2 clus. var.},
if $(Y,D)$ is positive and non-acyclic, then $D^{\perp}\simeq{\bf D}_{n}\,(n\geq4)$
or ${\bf E}_{n}\,(n=6,7,8)$. By Theorem 4.4 of \cite{TM}, in these
cases, we could assume that $(Y,D)$ has a toric model $\rho:(Y,D)\rightarrow(\mathbb{P}^{2},\overline{D})$
where $\overline{D}=\overline{D}_{1}+\overline{D}_{2}+\overline{D}_{3}$
is the toric boundary and $D=D_{1}+D_{2}+D_{3}$ is the strict transform
of $\overline{D}$. Denote by $l$ the class of the pullback of a
line in $\mathbb{P}^{2}$ via $\rho$. Let $\{E_{i}\}_{i\in I}$ be
the exceptional configuration of the toric model $\rho:(Y,D)\rightarrow\left(\mathbb{P}^{2},\overline{D}\right)$
as given on the left side of Figure \ref{fig:posacyc} and ${\bf s}=(e_{i})_{i\in I}$
the seed constructed from this toric model as described in subsection
\ref{subsec:Relationship-to-cluster}. We construct irreducible root
systems $\mathbf{D}_{n}$ ($n\geq4)$ or ${\bf E}_{n}$ ($n=6,7,8$)
explicitly as shown on the right side of the same figure.

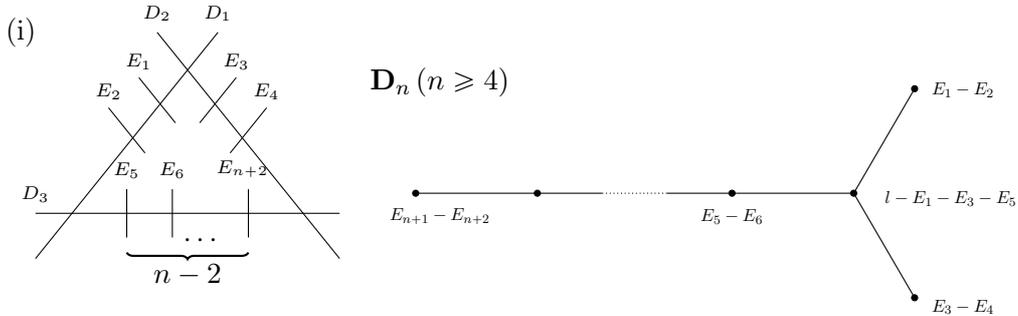
\begin{minipage}{\linewidth} \begin{tikzpicture} \begin{scope}[shift={(-3,0)},scale=0.4] \draw (1,5.5)--(-5,-2); \draw  (-1,5.5) -- (5,-2); \draw (-5,-0.5) -- (5,-0.5); \draw (-1.6,4) -- (-0.4,2.5); \draw (-2.6,3) -- (-1.4,1.5); \draw (1.6,4) -- (0.4,2.5); \draw (2.6,3) -- (1.4,1.5);  \draw (-2,0.3) -- (-2,-1.3); \draw (-0.5,0.3) -- (-0.5,-1.3); \draw (2,0.3) -- (2,-1.3); \node [label=below:$\cdots$] at (0.5,-0.5){}; \node [label=above:\tiny$E_{1}$] at (-1.6,3.6){}; \node [label=above:\tiny$E_{2}$] at (-2.6,2.6){}; \node [label=above:\tiny$E_{3}$] at (1.6,3.6){}; \node [label=above:\tiny$E_{4}$] at (2.6,2.6){}; \node [label=above:\tiny$E_{5}$] at (-2,0){}; \node [label=above:\tiny$E_{6}$] at (-0.5,0){}; \node [label=above:\tiny$E_{n+2}$] at (1.8,0){}; \node [label=above:\tiny$D_{2}$] at (-1,5.2){}; \node [label=above:\tiny$D_{1}$] at (1,5.2){}; \node [label=above:\tiny$D_{3}$] at (-5,-0.8){}; \node [label=below:$n-2$] at (0,-1.5){}; 
\draw [thick,decoration={brace, mirror,raise=0.2cm},decorate] (-2,-1.3) -- (2,-1.3); \node [label=left:(i)] at (-4.3,5.5) {}; \end{scope} 

\begin{scope}[scale=0.8] \dynkin[edge length=2cm]{D}{} \node  [scale =.6] at (0.4,-0.3){$E_{n+1}-E_{n+2}$}; \node  [scale =.6] at (5.2,-0.3){$E_{5}-E_{6}$}; \node  [scale =.6] at (8.8,0){$l-E_{1}-E_{3}-E_{5}$}; \node  [scale =.6] at (9,1.8){$E_{1}-E_{2}$}; \node  [scale =.6] at (9,-1.8){$E_{3}-E_{4}$}; \node [label=below:$\mathbf{D}_{n}\,(n\geq4)$] at (0.4,2.5){}; \end{scope} \end{tikzpicture} \end{minipage}\\\\
\begin{minipage}{\linewidth} \begin{tikzpicture} \begin{scope}[shift={(-3,0)},scale=0.4] \draw (1,5.5)--(-5,-2); \draw  (-1,5.5) -- (5,-2);  \draw (-5,-0.5) -- (5,-0.5); \draw (-1.6,4) -- (-0.4,2.5);  \draw (-2.6,3) -- (-1.4,1.5); \draw (1.6,4) -- (0.4,2.5);  \draw (2.6,3) -- (1.4,1.5);   \draw (3.6,2) -- (2.4,0.5); \draw (-1.5,0.3) -- (-1.5,-1.3);  \draw (0,0.3) -- (0,-1.3);  \draw (1.5,0.3) -- (1.5,-1.3); \node [label=above:\tiny$E_{1}$] at (-1.6,3.6){};  \node [label=above:\tiny$E_{2}$] at (-2.6,2.6){};  \node [label=above:\tiny$E_{3}$] at (1.6,3.6){};  \node [label=above:\tiny$E_{4}$] at (2.6,2.6){};  \node [label=above:\tiny$E_{5}$] at (3.6,1.6){};  \node [label=below:\tiny$E_{6}$] at (-1.5,-1){};  \node [label=below:\tiny$E_{7}$] at (0,-1){};  \node [label=below:\tiny$E_{8}$] at (1.5,-1){}; \node [label=above:\tiny$D_{2}$] at (-1,5.2){};  \node [label=above:\tiny$D_{1}$] at (1,5.2){};  \node [label=above:\tiny$D_{3}$] at (-5,-0.8){}; \node [label=left:(ii)] at (-4,5.5) {};
\end{scope}

\begin{scope}[scale=0.8] \dynkin[edge length=2cm]{E}{6} \node  [scale =.6] at (4,-0.3){$l-E_{1}-E_{3}-E_{6}$}; \node  [scale =.6] at (6,-0.3){$E_{7}-E_{6}$}; \node  [scale =.6] at (8,-0.3){$E_{8}-E_{7}$};
\node  [scale =.6] at (2,-0.3){$E_{4}-E_{3}$}; \node  [scale =.6] at (0,-0.3){$E_{5}-E_{4}$}; \node  [scale =.6] at (4.8,2){$E_{2}-E_{1}$}; \node [label=below:$\mathbf{E}_{6}$] at (0,2.5){}; \end{scope} \end{tikzpicture} \end{minipage}\\\\
\begin{minipage}{\linewidth} \begin{tikzpicture} \begin{scope}[shift={(-3,0)},scale=0.4] \draw (1,5.5)--(-5,-2); \draw  (-1,5.5) -- (5,-2); \draw (-5,-0.5) -- (5,-0.5);
\draw (-1.6,4) -- (-0.4,2.5); \draw (-2.6,3) -- (-1.4,1.5);
\draw (1.6,4) -- (0.4,2.5); \draw (2.4,3.2) -- (1.2,1.7);  \draw (3.2,2.4) -- (2,0.9);
\draw (-2.3,0.3) -- (-2.3,-1.3); \draw (-0.8,0.3) -- (-0.8,-1.3); \draw (0.7,0.3) -- (0.7,-1.3); \draw (2.2,0.3) -- (2.2,-1.3);
\node [label=above:\tiny$E_{1}$] at (-1.6,3.6){}; \node [label=above:\tiny$E_{2}$] at (-2.6,2.6){}; \node [label=above:\tiny$E_{3}$] at (1.5,3.6){}; \node [label=above:\tiny$E_{4}$] at (2.4,2.8){}; \node [label=above:\tiny$E_{5}$] at (3.4,1.8){}; \node [label=below:\tiny$E_{6}$] at (-2.3,-1){}; \node [label=below:\tiny$E_{7}$] at (-0.8,-1){}; \node [label=below:\tiny$E_{8}$] at (0.7,-1){}; \node [label=below:\tiny$E_{9}$] at (2.2,-1){};
\node [label=above:\tiny$D_{2}$] at (-1,5.2){}; \node [label=above:\tiny$D_{1}$] at (1,5.2){}; \node [label=above:\tiny$D_{3}$] at (-5,-0.8){}; \node [label=left:(iii)] at (-4,5.5) {};
\end{scope}
\begin{scope}[scale=0.8] \dynkin[edge length=2cm]{E}{7} \node  [scale =.6] at (4,-0.3){$l-E_{1}-E_{3}-E_{6}$}; \node  [scale =.6] at (6,-0.3){$E_{7}-E_{6}$}; \node  [scale =.6] at (8,-0.3){$E_{8}-E_{7}$}; \node  [scale =.6] at (10,-0.3){$E_{9}-E_{8}$}; \node  [scale =.6] at (2,-0.3){$E_{4}-E_{3}$}; \node  [scale =.6] at (0,-0.3){$E_{5}-E_{4}$}; \node  [scale =.6] at (4.8,2){$E_{2}-E_{1}$}; \node [label=below:$\mathbf{E}_{7}$] at (0,2.5){}; \end{scope} \end{tikzpicture} \end{minipage}\\\\
\begin{minipage}{\linewidth} \begin{tikzpicture} \begin{scope}[shift={(-3,0)},scale=0.4] \draw (1,5.5)--(-5,-2); \draw  (-1,5.5) -- (5,-2); \draw (-5,-0.5) -- (5,-0.5);
\draw (-1.6,4) -- (-0.4,2.5); \draw (-2.6,3) -- (-1.4,1.5);
\draw (1.6,4) -- (0.4,2.5); \draw (2.4,3.2) -- (1.2,1.7);  \draw (3.2,2.4) -- (2,0.9);
\draw (-2.4,0.3) -- (-2.4,-1.3); \draw (-1.2,0.3) -- (-1.2,-1.3); \draw (0,0.3) -- (0,-1.3); \draw (1.2,0.3) -- (1.2,-1.3); \draw (2.4,0.3) -- (2.4,-1.3);
\node [label=above:\tiny$E_{1}$] at (-1.6,3.6){}; \node [label=above:\tiny$E_{2}$] at (-2.6,2.6){}; \node [label=above:\tiny$E_{3}$] at (1.5,3.6){}; \node [label=above:\tiny$E_{4}$] at (2.4,2.8){}; \node [label=above:\tiny$E_{5}$] at (3.4,1.8){}; \node [label=below:\tiny$E_{6}$] at (-2.4,-1){}; \node [label=below:\tiny$E_{7}$] at (-1.2,-1){}; \node [label=below:\tiny$E_{8}$] at (0,-1){}; \node [label=below:\tiny$E_{9}$] at (1.2,-1){}; \node [label=below:\tiny$E_{10}$] at (2.4,-1){};
\node [label=above:\tiny$D_{2}$] at (-1,5.2){}; \node [label=above:\tiny$D_{1}$] at (1,5.2){}; \node [label=above:\tiny$D_{3}$] at (-5,-0.8){};
\node [label=left:(iv)] at (-4,5.5) {}; \end{scope}
\begin{scope}[scale=0.8] \dynkin[edge length=2cm]{E}{8} \node  [scale =.6] at (4,-0.3){$l-E_{1}-E_{3}-E_{6}$}; \node  [scale =.6] at (6,-0.3){$E_{7}-E_{6}$}; \node  [scale =.6] at (8,-0.3){$E_{8}-E_{7}$}; \node  [scale =.6] at (10,-0.3){$E_{9}-E_{8}$}; \node  [scale =.6] at (12,-0.3){$E_{10}-E_{9}$}; \node  [scale =.6] at (2,-0.3){$E_{4}-E_{3}$}; \node  [scale =.6] at (0,-0.3){$E_{5}-E_{4}$}; \node  [scale =.6] at (4.8,2){$E_{2}-E_{1}$};
\node [label=below:$\mathbf{E}_{8}$] at (0,2.5){}; \end{scope}
\end{tikzpicture} 
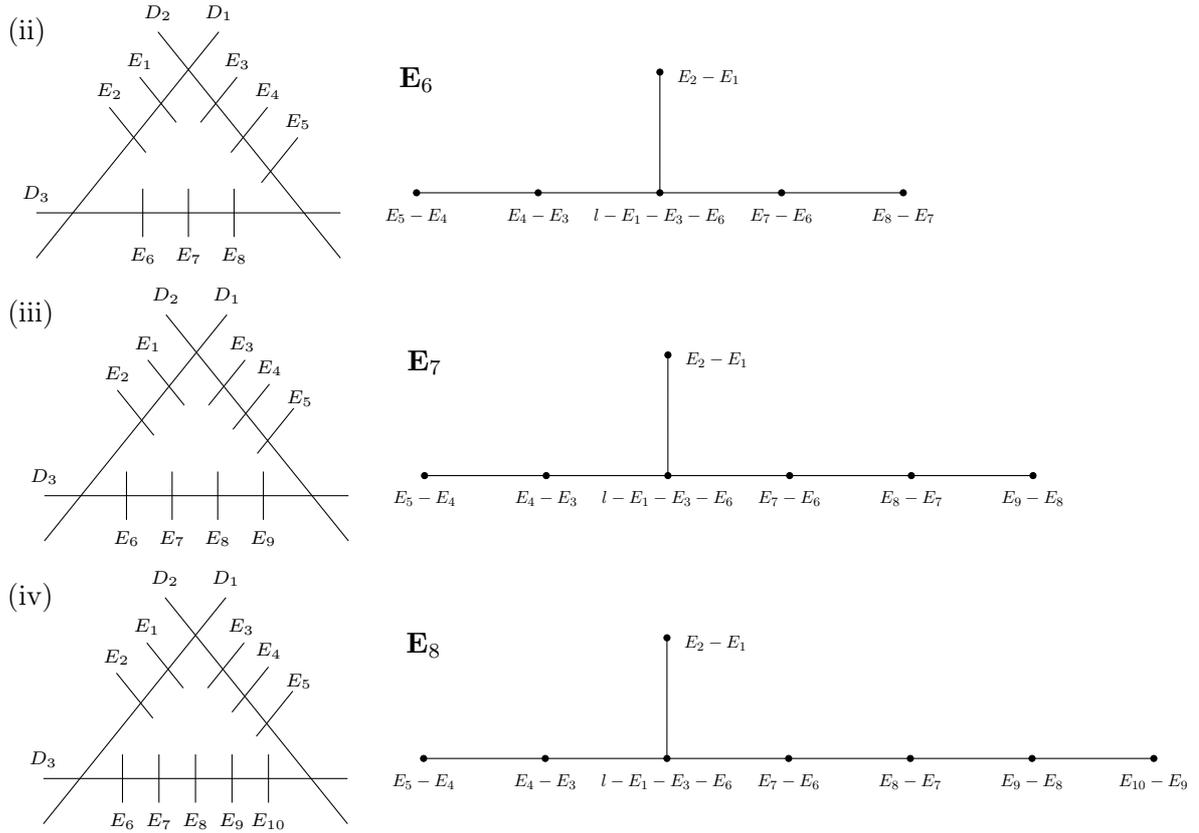
\captionof{figure}{$(Y,D)$ when $D^{\perp}\simeq{\mathbf{D}_{n}}\,(n\geq4)$ or $\mathbf{E}_{n}\,(n=6,7,8)$} \label{fig:posacyc} \end{minipage}\\\\

\subsection{Fundamental example: Blowing up $\mathbb{P}^{2}$ at three points.}

Consider the seed ${\bf s}$ given following exchange matrix \footnote{The exchange matrix we use here follows the convention of \cite{FZ07}.
To agree with the convention of \cite{GHKK18}, one need to replace
$\epsilon$ with $-\epsilon$. The two papers use different conventions
to define $\mathcal{X}$-cluster varieties. When we define the $\mathcal{X}$-cluster
varieties, we follows the convention of \cite{GHKK18}, but when we
do calculations we follow the convention of \cite{FZ07} because we
use Sage, which follows \cite{FZ07}, to do all computations.}
\[
\epsilon=\begin{pmatrix}0 & -1 & 1\\
1 & 0 & -1\\
-1 & 1 & 0
\end{pmatrix}.
\]
The element $e_{1}+e_{2}+e_{3}$ generates the kernel $K$ of the
skew-symmetric form. The corresponding $\mathcal{X}$-cluster variety
has the fibration $\mathcal{X}\rightarrow T_{K^{*}}$ such that a
generic fiber, up to codimension $2$, is a log Calabi-Yau surface
obtained by blowing up $\mathbb{P}^{2}$ at a smooth point on on each
toric boundary divisor $\overline{D}_{i}$ and then deleting the strict
transform $D$ of $\overline{D}$. After the identification of $K$
with $D^{\perp}$, $e_{1}+e_{2}+e_{3}$ corresponds to the root $l-E_{1}-E_{2}-E_{3}$
where $l$ is the pullback of a line in $\mathbb{P}^{2}$ via the
blowup and $E_{i}$s are the three exceptional divisors. Let 
\[
{\bf s}^{'}=\mu_{1}\circ\mu_{3}\circ\mu_{2}\circ\mu_{1}({\bf s}).
\]
The seed ${\bf s}^{'}=(e_{i}^{'})_{i\in I}$ has exchange matrix 
\[
\epsilon^{'}=\begin{pmatrix}0 & 1 & -1\\
-1 & 0 & 1\\
1 & -1 & 0
\end{pmatrix}
\]
and its $\mathbf{g}$-matrix is
\[
\begin{pmatrix} &  & -1\\
 & -1\\
-1
\end{pmatrix}
\]
So up to seed isomorphisms, ${\bf s}$ and ${\bf s}^{'}$ are related
by the Donaldson-Thomas transformation. Under $\mu_{{\bf s},{\bf s}^{'}}^{t}:M_{\mathbb{R}}\rightarrow M_{\mathbb{R}}$
where $\mu_{{\bf s,{\bf s^{'}}}}^{t}$ is the geometric tropicalization
of $\mu_{{\bf s,{\bf s}^{'}}}:T_{M,{\bf s}}\dashrightarrow T_{M,{\bf s}^{'}}$,
$-\{e_{1},\cdot\}$ is mapped to $-\{e_{1}^{'},\cdot\}$, $-\{e_{2}^{'},\cdot\}$
to $-\{e_{3}^{'},\cdot\}$ and $-\{e_{3}^{'},\cdot\}$ to $-\{e_{2}^{'},\cdot\}$.
Thus we define the seed isomorphism 
\begin{align*}
\gamma:{\bf s}^{'} & \rightarrow{\bf s},\,\,e_{1}^{'}\mapsto e_{1},\,\,e_{2}^{'}\mapsto e_{3},\,\,e_{3}^{'}\mapsto e_{2}
\end{align*}
Let $\delta=\gamma\circ\mu_{1}\circ\mu_{3}\circ\mu_{2}\circ\mu_{1}$.
Then $\delta$ is an element in the cluster modular group of $\mathcal{X}_{{\bf s}}$
and it induces a birational automorphism of $T_{N,{\bf s}}$ such
that 
\[
{\rm \delta}^{*}\left(z^{e_{1}}\right)=-z^{e_{3}}\frac{F_{1}}{F_{2}},\,\,{\rm \delta}^{*}\left(z^{e_{2}}\right)=-z^{e_{1}}\frac{F_{2}}{F_{3}},\,\,{\rm \delta}^{*}\left(z^{e_{3}}\right)=-z^{e_{2}}\frac{F_{3}}{F_{1}}
\]
where $F_{i}=f(z^{e_{i}},z^{e_{i+2}})$ with $f(z_{1},z_{2})=1+z_{1}+z_{1}z_{2}$.
Split $N$ as $N/K\oplus K$ via the following section: 
\begin{align*}
N/K & \hookrightarrow N\\
\overline{e_{1}} & \mapsto e_{1}\\
\overline{e_{3}} & \mapsto e_{3}.
\end{align*}
Then the above splitting induces an isomorphism of lattices $M\simeq K^{\perp}\oplus K^{*}$
and hence an isomorphism between toric varieties
\begin{align*}
{\rm TV}\left(\Sigma_{{\bf s}}\right) & \stackrel{\sim}{\rightarrow}\mathbb{P}^{2}\times\mathbb{C}^{*}
\end{align*}
Thus, we could write $\delta$ as a birational automorphism of toric
variety $\mathbb{P}^{2}\times\mathbb{C}^{*}$:
\begin{align*}
{\rm \delta}:\mathbb{P}^{2}\times\mathbb{C}^{*} & \dashrightarrow\mathbb{P}^{2}\times\mathbb{C}^{*}\\
\left(\left[z^{\overline{e_{3}}}:1:z^{-\overline{e_{1}}}\right],\alpha\right) & \longmapsto\left(\left[\alpha^{-1}z^{\overline{e_{1}}+\overline{e_{3}}}\frac{\overline{F}_{3}}{\overline{F_{1}}}:1:z^{\overline{e_{3}}}\frac{\overline{F}_{2}}{\overline{F}_{1}}\right],\alpha^{-1}\right)
\end{align*}
where $\alpha=z^{e_{1}+e_{2}+e_{3}}$, $\overline{F}_{1}=1+z^{\overline{e_{1}}}+z^{\overline{e_{1}}+\overline{e_{3}}}$,
$\overline{F}_{2}=1+\alpha z^{-\overline{e_{1}}-\overline{e_{3}}}+\alpha z^{-\overline{e_{3}}}$
and $\overline{F}_{3}=1+z^{\overline{e_{3}}}+\alpha z^{-\overline{e_{1}}}$.
Homogenize our equation by setting $z^{\overline{e_{3}}}=\frac{X_{1}}{X_{2}}$
and $z^{-\overline{e_{1}}}=\frac{X_{3}}{x_{2}}$, we get 
\begin{align*}
\left[\left(X_{1}:X_{2}:X_{3}\right),\alpha\right]\longmapsto & ([\alpha^{-1}X_{1}\left(X_{1}+X_{2}+\alpha X_{3}\right):X_{2}\left(X_{1}+X_{2}+X_{3}\right)\\
 & :X_{3}\left(X_{1}+\alpha X_{2}+\alpha X_{3}\right)]
\end{align*}
For $\alpha\neq1$, $\delta$ restricts to a birational automorphism
of $\mathbb{P}^{2}$:
\begin{align*}
\delta_{\alpha}:\mathbb{P}^{2}\dashrightarrow & \mathbb{P}^{2}\\
\left[X_{1}:X_{2}:X_{3}\right]\longmapsto & ([\alpha^{-1}X_{1}\left(X_{1}+X_{2}+\alpha X_{3}\right):X_{2}\left(X_{1}+X_{2}+X_{3}\right)\\
 & :X_{3}\left(X_{1}+\alpha X_{2}+\alpha X_{3}\right)]
\end{align*}
We can minimally resolve $\delta_{\alpha}$ as shown in the following
diagram, where $\rho_{1}:Y_{\alpha}\rightarrow\mathbb{P}^{2}$ is
the blowup of $\mathbb{P}^{2}$ at $p_{1}=[0:1:-1]$ , $p_{2}=[1:0:-\alpha^{-1}]$
and $p_{3}=\left[1:-1:0\right]$.

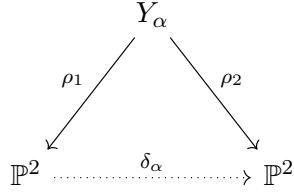
 \begin{minipage}{\linewidth} \begin{center}  \begin{tikzcd} [row sep=1.5cm]& Y_{\alpha} \arrow[dl,"\rho_{1}"'] \arrow[dr,"\rho_{2}"] & \\   \mathbb{P}^{2} \arrow[rr, dotted, "\delta_{\alpha}"] & &  \mathbb{P}^{2}   \end{tikzcd}   \end{center} \end{minipage} \\\\To
understand the morphism $\rho_{2}:Y_{\alpha}\rightarrow\mathbb{P}^{2}$,
we can look at the inverse image of each boundary divisor of $\mathbb{P}^{2}$
under ${\rm \delta}_{\alpha}$, as shown in the following figure.
Let $L_{ij}(1\leq i<j\leq3)$ be the line in $\mathbb{P}^{2}$ joining
$p_{i}$ and $p_{j}$. Then $L_{23}$ is given by the equation $X_{1}+X_{2}+\alpha X_{3}=0$,
$L_{13}$ by the equation $X_{1}+X_{2}+X_{3}=0$ and $L_{12}$ by
the equation $X_{1}+\alpha X_{2}+\alpha X_{3}=0$. After blowing up
at $p_{1},p_{2},$ and $p_{3}$, the strict transform of $L_{12},L_{23},L_{13}$
gives another exceptional configuration of $Y_{\alpha}$ and $\rho_{2}$
is the morphism that blows down this exceptional configuration. Therefore
$\delta$ coincides with the reflection with respect to the root $l-E_{1}-E_{2}-E_{3}$.\\

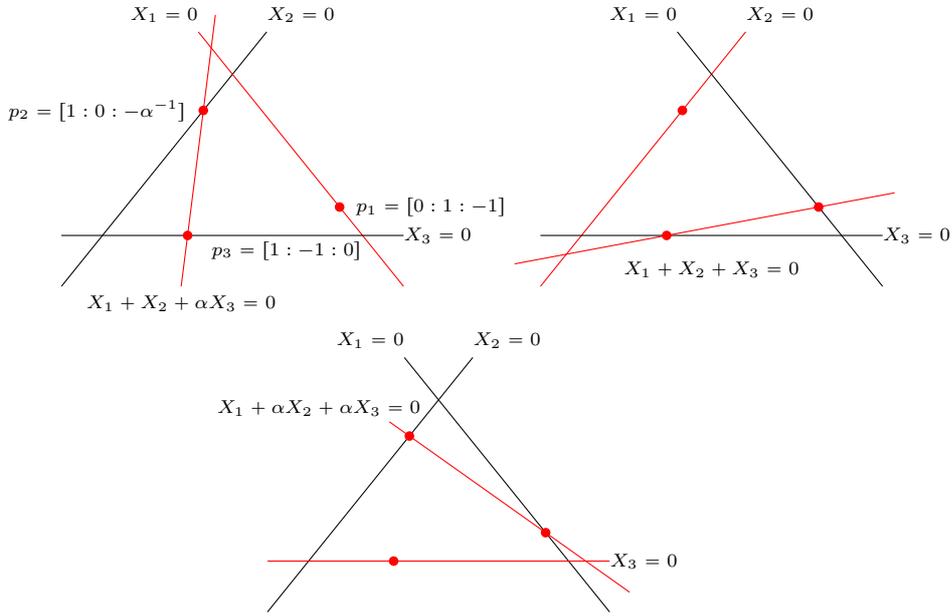
\begin{minipage}{\linewidth} \begin{center}  \begin{tikzpicture}[scale=0.5] \begin{scope}[shift={(-14,0)}] \draw (1,5.5)--(-5,-2); \draw [red] (-1,5.5) -- (5,-2); \draw (-5,-0.5) -- (5,-0.5); \draw (6,-0.5) node {\tiny{$X_{3}=0$}}; \draw (2,6) node {\tiny{$X_{2}=0$}}; \draw (-2,6) node {\tiny{$X_{1}=0$}}; \path [name path=line 1] (-0.5,6) -- (-1.5,-2); \path [name path=line a] (1,5.5)--(-5,-2); \path [name path=line b] (-5,-0.5) -- (5,-0.5); \path [name path=line c] (-1,5.5) -- (5,-2); \draw [red] (-0.5,6) -- (-1.5,-2); \draw (-1.5,-2.5) node {\tiny{$X_{1}+X_{2}+\alpha X_{3}=0$}}; \path [name intersections={of= line 1 and line a, by=x}]; \fill [red] (x) circle (4pt); \path [name intersections={of= line 1 and line b, by=y}]; \fill [red] (y) circle (4pt); \path [name path=line 2]  (x) -- (5,-1); \path [name intersections={of= line 2 and line c, by=z}]; \fill [red] (z) circle (4pt); \draw (x) node [xshift=-1.4cm] {\tiny$p_{2}=[1:0:-\alpha^{-1}]$}; \draw (z) node [xshift=1.2cm]{\tiny$p_{1}=[0:1:-1]$}; \draw (y) node [xshift=1.3cm,yshift=-0.2cm]{\tiny$p_3=[1:-1:0]$}; \end{scope}    \begin{scope}  \draw [red](1,5.5)--(-5,-2); \draw  (-1,5.5) -- (5,-2); \draw (-5,-0.5) -- (5,-0.5); \draw (6,-0.5) node {\tiny{$X_{3}=0$}}; \draw (2,6) node {\tiny{$X_{2}=0$}}; \draw (-2,6) node {\tiny{$X_{1}=0$}}; \path [name path=line 1] (-0.5,6) -- (-1.5,-2); \path [name path=line a] (1,5.5)--(-5,-2); \path [name path=line b] (-5,-0.5) -- (5,-0.5); \path [name path=line c] (-1,5.5) -- (5,-2); \draw (0,-1.5) node {\tiny{$X_{1}+X_{2} + X_{3}=0$}}; \path [name intersections={of= line 1 and line a, by=x}]; \fill [red] (x) circle (4pt); \path [name intersections={of= line 1 and line b, by=y}]; \fill [red] (y) circle (4pt); \path [name path=line 2]  (x) -- (5,-1); \path [name intersections={of= line 2 and line c, by=z}]; \fill [red] (z) circle (4pt); \draw [red, add=1 and 0.5] (y) to (z); \end{scope}\end{tikzpicture} \end{center} \end{minipage}\\\\
 \begin{minipage}{\linewidth} \begin{center}  \begin{tikzpicture}[scale=0.5] \draw (1,5.5)--(-5,-2); \draw  (-1,5.5) -- (5,-2); \draw [red](-5,-0.5) -- (5,-0.5); \draw (6,-0.5) node {\tiny{$X_{3}=0$}}; \draw (2,6) node {\tiny{$X_{2}=0$}}; \draw (-2,6) node {\tiny{$X_{1}=0$}}; \path [name path=line 1] (-0.5,6) -- (-1.5,-2); \path [name path=line a] (1,5.5)--(-5,-2); \path [name path=line b] (-5,-0.5) -- (5,-0.5); \path [name path=line c] (-1,5.5) -- (5,-2); \draw (-3.5,4) node {\tiny{$X_{1}+ \alpha X_{2} + \alpha X_{3}=0$}}; \draw [name intersections={of= line 1 and line a, by=x}] [red, add=.1 and .1] (x) to (5,-1); \fill [red] (x) circle (4pt); \path [name intersections={of= line 1 and line b, by=y}]; \fill [red] (y) circle (4pt); \path [name path=line 2]  (x) -- (5,-1); \path [name intersections={of= line 2 and line c, by=z}]; \fill [red] (z) circle (4pt);   \end{tikzpicture} \end{center} \end{minipage}\\
\begin{defn}
Given a Looijenga pair $(Y,D)$ and a root $\alpha$ in $D^{\perp}$,
denote by $r_{\alpha}$ the reflection automorphism of ${\rm Pic}(Y)$
with respect to the root $\alpha$.

Throughout the rest of this section, $(Y,D)$ will be a generic Looijenga
pair with $D^{\perp}\simeq{\bf D}_{n}(n\geq4)$ or ${\bf E}_{n}(n=6,7,8$)
and ${\bf s}$ the seed constructed from the toric model as shown
in Figure \ref{fig:posacyc}. We also use $\mathcal{X}_{\mathbf{D}_{n}}(n\geq4)$
or $\mathcal{X}_{\mathbf{E}_{n}}(n=6,7,8)$ to denote the corresponding
$\mathcal{X}$-cluster variety when we want to emphasize the particular
deformation type of $U=Y\setminus D$. Let $\mathfrak{D}_{{\bf s}}$
be the scattering diagram corresponding to the seed ${\bf s}$. Since
${\bf s}$ is skew-symmetric, $\mathcal{A}_{{\bf s}}^{\vee}\simeq\mathcal{X}_{{\bf s}}$.
Hence the scattering diagram $\mathfrak{D}_{{\bf s}}$ lives on $\mathcal{X}_{{\bf s}}(\mathbb{R}^{T})$.
Moreover, we have the following proposition:
\end{defn}

\begin{prop}
\label{prop:Weyl group preseves cluster complex}The Weyl group $W(Y,D)$
of $(Y,D)$ acts faithfully on the scattering diagram $\mathfrak{D}_{{\bf s}}$
as a subgroup of the cluster modular group. In particular, its action
preserves the cluster complex in $\mathfrak{D}_{{\bf s}}$. 
\end{prop}

\begin{proof}
For simple roots of the form $E_{i_{1}}-E_{i_{2}}$ where $E_{i_{1}}$
and $E_{i_{2}}$ are two distinct exceptional curves along the same
boundary component of $D$, we have $\{e_{i_{1}},\cdot\}=\{e_{i_{2}},\cdot\}$.
The reflection $r_{E_{i_{1}}-E_{i_{2}}}$ acts on the scattering diagram
linearly by interchanging the dual basis elements $e_{i_{1}}^{*}$
and $e_{i_{2}}^{*}$. For simple roots of the form $l-E_{\epsilon_{1}}-E_{\epsilon_{2}}-E_{\epsilon_{3}}$
where $E_{\epsilon_{i}}$ is an exceptional curve along $\overline{D}_{i}$,
we have seen in the fundamental example that the reflection with respect
these roots are also elements in the cluster modular group. Since
the simple roots form a root basis for $D^{\perp}$, $W(Y,D)$ is
generated by reflections with respect to simple roots. Hence we conclude
that the action of $W(Y,D)$ preserves the cluster complex.
\end{proof}
Here is the outline for the rest of this section. In subsection \ref{subsec:Factorizations-of-DT-transformat}
and subsection \ref{subsec:Factorizations-of-DT-transformat-1}, we
will find a special element ${\bf w}$\footnote{Since cluster transformations are unique up to seed automorphisms,
${\bf w}$ is unique up to compositions by seed automorphisms of ${\bf s}$,
which are generated by reflections by simple roots of the form $E_{i_{1}}-E_{i_{2}}$
where $E_{i_{1}}$ and $E_{i_{2}}$ are two distinct exceptional curves
along the same boundary component of $D$.} in the Weyl group of $W(D^{\perp})$ with $D^{\perp}\simeq\mathbf{D}_{n}(n\geq4)$
or $\mathbf{E}_{n}(n=6,7,8)$ such that the action of ${\bf w}$ on
$\mathcal{X}_{D^{\perp}}$ is close to the Donaldson-Thomas transformation
in the cluster modular group. As a result, we could approximate ${\rm DT}_{\mathcal{X}_{D^{\perp}}}$
by ${\bf w}$ and factorize ${\rm DT}_{\mathcal{X}_{D^{\perp}}}$
into cluster transformations since the action of ${\bf w}$ on $\mathcal{X}_{D^{\perp}}$
is cluster. For details of the Donaldson-Thomas transformations of
cluster varieties, see \cite{GS16}. In subsection \ref{subsec:Application-of-folding},
we apply the folding procedure as described in section \ref{sec:The-geometry-of}
to ${\bf s}$, $\mathcal{X}_{{\bf s}}$ and $\mathfrak{D}_{{\bf s}}$.
The main result will be Theorem \ref{thm:The-scattering-diagram}. 

\subsection{\label{subsec:Factorizations-of-DT-transformat}Factorizations of
DT-transformations as cluster transformations for $\mathcal{X}_{\mathbf{D}_{n}}\,(n\protect\geq4)$.}

Suppose $D^{\perp}\simeq\mathbf{D}_{n}\,(n\geq4)$. For any $5\leq i\leq n+2$,
denote by ${\bf w}_{i}$ the composition of reflections 
\begin{align}
{\bf w}_{i} & =r_{l-E_{1}-E_{3}-E_{i}}\circ r_{l-E_{2}-E_{4}-E_{i}}.\label{eq:-3_w_i}
\end{align}
Let ${\bf s}_{w}=(e_{i}^{w})_{i\in I}$ be the seed we get after performing
the following sequence of mutations
\begin{equation}
[5,1,3,2,4,5,\,\,6,1,3,2,4,6,\,\,\cdots\,\,n+2,1,3,2,4,n+2]\quad\left(n\geq4\right).\label{eq:-2}
\end{equation}
We denote this sequence of mutations by $W_{n}$. Let $\gamma:{\bf s}_{w}\rightarrow{\bf s}$
be the seed isomorphism such that if $n$ is even, $\gamma(e_{i}^{w})=e_{i}$
for all $i$ and if $n$ is odd, $\gamma(e_{1}^{w})=e_{3},$ $\gamma(e_{2}^{w})=e_{4}$,
$\gamma(e_{i}^{w})=e_{i}$ for all $i\geq5$. It follows from the
mutation formula for a single reflection of $r_{l-E_{\epsilon_{1}}-E_{\epsilon_{2}}-E_{\epsilon_{3}}}$
($\epsilon_{1}\in\{1,2\}$, $\epsilon_{2}\in\{3,4\}$ and $\epsilon_{3}\in\{5,6,\cdots,n+1,n+2\}$)
that the action of 
\begin{equation}
{\bf w}_{5}\circ{\bf w}_{6}\circ\cdots\circ{\bf w}_{n+2}\label{eq:-3 D_n}
\end{equation}
is given by $\gamma\circ W_{n}$.
\begin{lem}
\label{lem:weyl gp elem approximating DT} For any $5\leq i,j\leq n+2$
($n\geq4$), ${\bf w}_{i}$ and ${\bf w}_{j}$ commute with each other.
In particular, the Weyl group element ${\bf w}_{5}\circ{\bf w}_{6}\circ\cdots\circ{\bf w}_{n+2}$
is of order $2$.
\end{lem}

\begin{proof}
For any $5\leq i\leq n+2$, since
\[
\left\langle l-E_{1}-E_{3}-E_{i},\,l-E_{2}-E_{4}-E_{i}\right\rangle =0,
\]
the reflections $r_{l-E_{1}-E_{3}-E_{i}}$ and $r_{l-E_{2}-E_{4}-E_{i}}$
commute. Therefore, ${\bf w}_{i}$ is of order $2$ for any $5\leq i\leq n+2$.
For any $5\leq i,j\leq n+2$ with $i\neq j$, 
\begin{align*}
{\bf w}_{j}{\bf w}_{i}{\bf w}_{j} & =\left({\bf w}_{j}\circ r_{l-E_{1}-E_{3}-E_{i}}\circ{\bf w}_{j}\right)\left({\bf w}_{j}\circ r_{l-E_{2}-E_{4}-E_{i}}\circ{\bf w}_{j}\right)\\
 & =r_{l-E_{2}-E_{4}-E_{i}}\circ r_{l-E_{1}-E_{3}-E_{i}}\\
 & ={\bf w}_{i}
\end{align*}
Therefore, ${\bf w}_{i}$ and ${\bf w}_{j}$ commute for any $5\leq i,j\leq n+2$.
Together with that ${\bf w}_{i}$ is of order $2$ for any $5\leq i\leq n+2$,
we conclude that ${\bf w}_{5}\circ{\bf w}_{6}\circ\cdots\circ{\bf w}_{n+2}$
is of order $2$.
\end{proof}
Let $\Pi$ be the subgroup of the symmetric group $S_{n+2}$ generated
by involutions 
\[
(1\,\,2),(3\,\,4),(5\,\,6),(6\,\,7),\cdots,(n+1\,\,n+2)
\]
and $\Pi^{'}$ the subgroup of $\Pi$ generated by involutions $(1\,\,2)$
and $(3\,\,4)$.
\begin{lem}
\label{lem:c-matrix-reflection}After the sequence of mutations $W_{n}$
(\ref{eq:-2}), the $\mathbf{c}$-matrix $C^{w}$ of ${\bf s}_{w}$
is of the form 
\[
\begin{pmatrix}M_{n}\\
 & -I_{n-2}
\end{pmatrix}
\]
where $M_{n}$ is a $4\times4$ matrix and $I_{n-2}$ is the identity
matrix of rank $n-2$.
\end{lem}

\begin{proof}
We prove this statement by induction on $n$. When $n=4$, after the
sequence of mutations $W_{4}$, the $\mathbf{c}$-matrix is {\scriptsize{}
\[
\begin{pmatrix} & -1\\
-1\\
 &  &  & -1\\
 &  & -1\\
 &  &  &  & -1\\
 &  &  &  &  & -1
\end{pmatrix}.
\]
}So the statement holds for the base case. Suppose for $n=k(k\geq4)$,
the statement is true. When $n=k+1$, by induction hypothesis, after
the sequence of mutations $W_{k}$, the ${\bf c}$-matrix is of the
form
\[
C^{\alpha}=\begin{pmatrix}M_{k} &  & *\\
 & -I_{k-2} & *\\
* & * & *
\end{pmatrix}
\]
Here $M_{k}$ is a $4\times4$ matrix and we do not know yet what
the last column and the last row of the $\mathbf{c}$-matrix $C^{\alpha}$
are. We claim that the last row of $C^{\alpha}$ has $1$ at the last
entry and $0$ at all other entries; the last column of $C^{\alpha}$
has $1$ at the first four and the last entry and $0$ at all other
entries. Indeed, by the mutation formula for ${\bf c}$-vectors, since
we have not mutated in the $(k+3)$-th direction yet, the last row
of $C^{\alpha}$ must have all but the last entries equal to $0$.
Denote the $i^{th}$-column of the ${\bf c}$-matrix $C^{\alpha}$
by $C_{i}^{\alpha}$. To see that the last column $C_{k+3}^{\alpha}$
is of the desired form, recall that for the initial seed ${\bf s}=\left(e_{i}\right)_{i\in I}$,
under the identification $D^{\perp}\simeq K$ as given in Theorem
5.5 of \cite{GHK15}, $e_{k+3}-e_{5}$ is identified with the root
$E_{k+3}-E_{5}$ and the root 
\[
{\bf w}_{5}\circ{\bf w}_{6}\circ\cdots\circ{\bf w}_{k+2}(E_{k+3}-E_{5})=2l-\sum_{i=1}^{5}E_{i}-E_{k+3}
\]
is identified with $\sum_{i=1}^{5}e_{i}+e_{k+3}$ in $K$. So we know
that

\[
C_{k+3}^{\alpha}-C_{5}^{\alpha}=\sum_{i=1}^{5}e_{i}+e_{k+3}.
\]
Since $C_{5}^{\alpha}=-e_{5}$, we know that $C_{k+3}^{\alpha}=e_{1}+e_{2}+e_{3}+e_{4}+e_{k+3}$
. Thus, if we continue to perform the last reflection ${\bf w}_{k+3}$,
the ${\bf c}$-matrix will be of the form 
\[
C^{\beta}=\begin{pmatrix}M_{k+1} &  & *\\
 & -I_{k-2} & *\\
 &  & *
\end{pmatrix}
\]
Again, we consider the action of the Weyl group on $D^{\perp}$. The
root 
\[
{\bf w}_{5}\circ{\bf w}_{6}\circ\cdots\circ{\bf w}_{k+2}\circ{\bf w}_{k+3}(E_{k+3}-E_{5})=E_{5}-E_{k+3}
\]
is identified with $e_{5}-e_{k+3}$, so 
\[
C_{k+3}^{\beta}-C_{5}^{\beta}=e_{5}-e_{k+3}.
\]
Since $C_{5}^{\beta}=-e_{5}$, we have $C_{k+3}^{\beta}=-e_{k+3}$.
Hence the ${\bf c}$-matrix is of the desired form for $n=k+3$. Thus
the induction is complete.
\end{proof}
\begin{thm}
\label{Thm:The-Donaldson-Thomas-transformat}The Donaldson-Thomas
transformation of $\mathcal{X}_{{\rm {\bf D}}_{n}}$ $(n\geq4)$ is
cluster.
\begin{proof}
By Lemma \ref{lem:Pi-invariant chambers}, after the sequence of mutations
as given in \ref{eq:-2}, the $\mathbf{c}$-matrix $C^{w}$ is of
the form 
\[
C^{w}=\begin{pmatrix}M_{n}\\
 & -I_{n-2}
\end{pmatrix}
\]
where $M_{n}$ is a $4\times4$ matrix and $I_{n-2}$ is the identity
matrix of rank $n-2$. By the relationship between $\mathbf{g}$-matrix
and $\mathbf{c}$-matrix as recorded in Lemma 5.12 of \cite{GHKK18},
we know that the $\mathbf{g}$-matrix of ${\bf s}_{w}$ is of the
form
\[
G^{w}=\begin{pmatrix}M_{n}^{'}\\
 & -I_{n-2}
\end{pmatrix}
\]
where $M_{n}^{'}$ is a $2\times2$ matrix and $I_{n-2}$ is the identity
matrix of rank $n-2$. Observe that the mutation sequence $W_{n}$
is equivalent to the mutation sequence 
\[
W_{n}^{'}:[5,1,2,3,4,5,\,\,6,1,2,3,4,6,\,\cdots,\,n+2,1,2,3,4,n+2]
\]
that is, $W_{n}$ and $W_{n}^{'}$ give rise to the same $\mathbf{c}$-matrix.
Let $\Pi^{'}$ be the subgroup of the symmetric group $S_{n+2}$ generated
by involutions $(1\,\,2)$ and $(3\,\,4)$. Since $W_{n}^{'}$ is
$\Pi^{'}$-constrained, by Lemma \ref{lem:Pi-invariant chambers},
$G^{w}$ is $\Pi^{'}$-invariant and hence $M_{n}^{'}$ is $\Pi^{'}$-invariant.
Let $\mathcal{C}_{w}^{+}$ be the cluster chamber corresponding to
the seed ${\bf s}_{w}$. Now, to show that Donaldson-Thomas transformation
for $\mathcal{X}_{{\rm {\bf D}}_{n}}$ is cluster, it remains to show
that the chamber $\mathcal{C}_{w}^{+}$ and and the chamber $\mathcal{C}_{{\bf s}}^{-}$
are mutationally equivalent. Since $M_{n}^{'}$ is $\Pi^{'}$-invariant,
by first projecting to the first $4$ coordinates of $M_{\mathbb{R}}$
and then applying the folding procedure, we are reduced to the case
of Kronecker $2$-quiver. Then the mutational equivalence between
$\mathcal{C}_{w}^{+}$ and $\mathcal{C}_{{\bf s}}^{-}$ follows from
the fact that the cluster complex of the scattering diagram of the
Kronecker $2$-quiver fill the whole plane except for a single ray.
\end{proof}
\end{thm}

\begin{rem}
\label{rem:It-follows-from}It follows from the proof of the above
theorem that, for $n\geq5$, the Weyl group element ${\bf w}$ and
${\rm DT}_{\mathcal{X}_{\mathbf{D}_{n}}}$ differs only by a finite
alternating sequence of mutations $1,2$ and $3,4$. In fact, one
can show that if $n\geq5$ and $n$ is odd, ${\bf w}$ and ${\rm DT}_{\mathcal{X}_{\mathbf{D}_{n}}}$
differs by 
\[
\left[1,2\,\,\,\underbrace{3,4,1,2,\,\,\,3,4,1,2,\,\,\,\cdots3,4,1,2}_{\left(\frac{n-5}{2}\right)\text{ iterations}}\right]
\]
and if $n\geq5$ and $n$ is even, ${\bf w}$ and ${\rm DT}_{\mathcal{X}_{\mathbf{D}_{n}}}$
differs by 
\[
\left[\underbrace{3,4,1,2,\,\,\,3,4,1,2,\,\,\,3,4,1,2}_{\left(\frac{n-4}{2}\right)\text{ iterations}}\right].
\]
\end{rem}

\subsection{\label{subsec:Factorizations-of-DT-transformat-1}Factorizations
of DT-transformations as cluster transformations for $\mathcal{X}_{{\bf E}_{n}}(n=6,7,8)$.}

\subsubsection{$D^{\perp}\simeq{\bf E}_{6}$.}

For $D^{\perp}\simeq\mathbf{E}_{6}$, let 
\begin{align}
\alpha_{1} & =2l-E_{1}-E_{2}-E_{3}-E_{4}-E_{6}-E_{7}\label{eq:-3}\\
\alpha_{2} & =2l-E_{1}-E_{2}-E_{3}-E_{5}-E_{6}-E_{8}\nonumber \\
\alpha_{3} & =2l-E_{1}-E_{2}-E_{4}-E_{5}-E_{7}-E_{8}.\nonumber 
\end{align}
Notice that for $1\leq i,j\leq3$, $r_{\alpha_{i}}$ and $r_{\alpha_{j}}$
commute. For each $r_{\alpha_{i}}$ , we could further factor $r_{\alpha_{i}}$
as reflections with respect to simple roots. For example we could
factor $r_{\alpha_{1}}$ as
\[
r_{\alpha_{1}}=r_{l-E_{1}-E_{3}-E_{6}}\circ r_{l-E_{1}-E_{4}-E_{7}}\circ r_{l-E_{2}-E_{3}-E_{6}}\circ r_{l-E_{2}-E_{4}-E_{7}}.
\]
The Weyl group element 
\begin{equation}
r_{\alpha_{1}}\circ r_{\alpha_{2}}\circ r_{\alpha_{3}}\label{eq: E_6}
\end{equation}
differs from the DT-transformation of $\mathcal{X}_{{\bf E}_{6}}$
by two mutations $[1,2]$. See Appendix \ref{sec:Mutation-sequences-for}
for an explicit mutation sequence for ${\rm DT}_{\mathcal{X}_{\mathbf{E}_{6}}}$.

\subsubsection{$D^{\perp}\simeq{\bf E}_{7}$.}

For $D^{\perp}\simeq\mathbf{E}_{7}$, let
\begin{align}
\alpha_{1} & =2l-E_{1}-E_{2}-E_{3}-E_{4}-E_{6}-E_{7}\label{eq:-4}\\
\alpha_{2} & =2l-E_{1}-E_{2}-E_{3}-E_{5}-E_{6}-E_{8}\nonumber \\
\alpha_{3} & =2l-E_{1}-E_{2}-E_{4}-E_{5}-E_{6}-E_{9}\nonumber \\
\alpha_{4} & =2l-E_{1}-E_{2}-E_{3}-E_{4}-E_{8}-E_{9}\nonumber \\
\alpha_{5} & =2l-E_{1}-E_{2}-E_{3}-E_{5}-E_{7}-E_{9}\nonumber \\
\alpha_{6} & =2l-E_{1}-E_{2}-E_{4}-E_{5}-E_{7}-E_{8}.\nonumber 
\end{align}
Notice that for $1\leq i,j\leq6$, $r_{\alpha_{i}}$ and $r_{\alpha_{j}}$
commute. The Weyl group element 
\begin{equation}
\prod_{i=1}^{6}r_{\alpha_{i}}\label{eq:E_7}
\end{equation}
 agrees with ${\rm DT}_{\mathcal{X}_{\mathbf{E}_{7}}}$. See Appendix
\ref{sec:Mutation-sequences-for} for an explicit mutation sequence
for ${\rm DT}_{\mathcal{X}_{\mathbf{E}_{7}}}$.

\subsubsection{$D^{\perp}\simeq{\bf E}_{8}$.}

For $D^{\perp}\simeq{\bf E}_{8}$, let 
\begin{equation}
\alpha=6l-3(E_{1}+E_{2})-2(E_{3}+E_{4}+E_{5})-(2E_{6}+E_{7}+E_{8}+E_{9}+E_{10})\label{eq:-5}
\end{equation}
\begin{align}
\beta_{1} & =2l-E_{1}-E_{2}-E_{3}-E_{4}-E_{7}-E_{8}\label{eq:-6}\\
\beta_{2} & =2l-E_{1}-E_{2}-E_{4}-E_{5}-E_{7}-E_{9}\nonumber \\
\beta_{3} & =2l-E_{1}-E_{2}-E_{3}-E_{5}-E_{7}-E_{10}\nonumber \\
\beta_{4} & =2l-E_{1}-E_{2}-E_{3}-E_{4}-E_{9}-E_{10}\nonumber \\
\beta_{5} & =2l-E_{1}-E_{2}-E_{4}-E_{5}-E_{8}-E_{10}\nonumber \\
\beta_{6} & =2l-E_{1}-E_{2}-E_{3}-E_{5}-E_{8}-E_{9}.\nonumber 
\end{align}
Notice that for any $1\leq i,j\leq6$, $r_{\beta_{i}}$ and $r_{\beta_{j}}$
commute and $r_{\alpha}$ commutes with $r_{\beta_{i}}$ for any $i$.
We could factor $r_{\alpha}$ as 
\[
r_{\alpha}=r_{\gamma_{1}}\circ r_{\delta_{1}}\circ r_{\gamma_{2}}\circ r_{\delta_{2}}
\]
where for $\epsilon\in\{1,2\}$, 
\begin{align*}
\gamma_{\epsilon} & =3l-\sum_{i=1}^{8}E_{i}-E_{\epsilon}\\
\delta_{\epsilon} & =3l-\sum_{i=1}^{6}E_{i}-(E_{9}+E_{10})-E_{\epsilon}.
\end{align*}
We could factor $r_{\gamma_{\epsilon}}$ and $r_{\delta_{\epsilon}}$
into compositions of reflections with respect to simple roots. For
example we could factor $r_{\gamma_{1}}$ as
\[
r_{\gamma_{1}}=r_{l-E_{1}-E_{5}-E_{8}}\circ r_{2l-E_{1}-E_{2}-E_{3}-E_{4}-E_{6}-E_{7}}\circ r_{l-E_{1}-E_{5}-E_{8}}
\]
and we know how to further factor $r_{2l-E_{1}-E_{2}-E_{3}-E_{4}-E_{6}-E_{7}}$.

The Weyl group element 
\begin{equation}
r_{\alpha}\circ\prod_{i=1}^{6}r_{\beta_{i}}\label{eq:-3 E_8}
\end{equation}
agrees with ${\rm DT}_{\mathcal{X}_{\mathbf{E}_{8}}}$. See Appendix
\ref{sec:Mutation-sequences-for} for an explicit mutation sequence
for ${\rm DT}_{\mathcal{X}_{\mathbf{E}_{8}}}$.

Combine results in this subsection, we have the following theorem:
\begin{thm}
\label{Thm:The-Donaldson-Thomas-transformat-1}The Donaldson-Thomas
transformation of $\mathcal{X}_{{\bf E}_{n}}$ ($n=6,7,8)$ is cluster.
\end{thm}

\begin{cor}
The Donaldson-Thomas transformation of $\mathcal{A}_{{\rm prin},{\bf s}}$
is cluster. 
\end{cor}

\begin{proof}
This follows immediately from Theorem \ref{Thm:The-Donaldson-Thomas-transformat},
Theorem \ref{Thm:The-Donaldson-Thomas-transformat-1} and the definition
of ${\rm DT}_{\mathcal{A}_{{\rm prin},{\bf s}}}$ (cf. \cite{GS16}). 
\end{proof}

\subsection{\label{subsec:Application-of-folding}Application of folding to cases
where $D^{\perp}\simeq\mathbf{D}_{n}\,(n\protect\geq4)\text{ or }\mathbf{E}_{n}(n=6,7,8)$.}

Let us recall some basic definitions of cluster structures for log
Calabi-Yau varieties.
\begin{defn}
\label{def:cluster structures}Given a log Calabi-Yau variety $V$,
if there exists a cluster variety $V^{'}=\bigcup_{w\in\mathfrak{T}_{{\bf s}_{0}}}T_{L,{\bf s}_{w}}$
of type $\mathcal{A}$, $\mathcal{X}$ or $\mathcal{A}_{{\rm prin}}$
together with a birational morphism
\[
\iota:V^{'}\rightarrow V
\]
such that for any pair of adjacent seeds ${\bf s}_{w}\stackrel{\mu_{k}}{\rightarrow}{\bf s}_{w^{'}}$
in $\mathfrak{T}_{{\bf s}_{0}}$, the restriction of $\iota$ to $T_{L,{\bf s}_{w}}\bigcup_{\mu_{k}}\allowbreak T_{L,{\bf s}_{w^{'}}}$
is an open embedding into $V$, then we say that the cluster torus
charts $T_{L,{\bf s}_{w}}$ together with their embeddings $\iota_{{\bf s}_{w}}:T_{L,{\bf s}_{w}}\hookrightarrow V$
into $V$ provide \emph{an atlas of cluster torus charts }for $V$.
If $\iota$ is an isomorphism outside strata of codimension at least
$2$ in the domain and range, then we say that the atlas of cluster
torus charts $\{(T_{L,{\bf s}_{w}},\iota_{{\bf s}_{w}})\}_{w\in\mathfrak{T}_{{\bf s}_{0}}}$
provides a \emph{cluster structure} for $V$. In this case, the codimension
$2$ condition guarantees that $V$ and the cluster variety $V^{'}$
have the same ring of regular functions.
\end{defn}

\begin{rem}
In the above definition, by calling $V$ a variety, we are assuming
that $V$ is an integral, separated scheme over $\mathbb{C}$, though
possibly not of finite type. As shown in \cite{GHK15}, $\mathcal{A}$
and $\mathcal{A}_{{\rm prin}}$ spaces are always separated while
$\mathcal{X}$ spaces are in general not (cf. Theorem 3.14 and Remark
4.2 of \cite{GHK15}). In this light, calling the $\mathcal{X}$ spaces
\emph{$\mathcal{X}$ cluster varieties} can be very misleading, though
the terminology has already been accepted in the cluster literature.
In the above definition, if $V^{'}$ is of type $\mathcal{A}$ or
$\mathcal{A}_{{\rm prin}}$, we could just assume that $\iota$ is
an open embedding. If $V^{'}$ is of type $\mathcal{X}$, by requiring
only the restriction of $\iota$ to any pair of adjacent cluster tori
to be an open embedding into $V$, we allow certain loci in $V^{'}$
to get identified when mapped into $V$ and therefore get around the
issue of non-separateness of $V^{'}$.
\end{rem}

\begin{defn}
\label{def:non-eqiuv-cluster}Given a log Calabi-Yau variety $V$
and two atlases of cluster torus charts 
\[
\mathcal{T}_{1}=\left\{ \left(T_{L,{\bf s}_{w}},\iota_{{\bf s}_{w}}\right)\right\} _{w\in\mathfrak{T}_{{\bf s}_{0}}},\quad\mathcal{T}_{2}=\left\{ \left(T_{L^{'},{\bf s}_{v}},\iota_{{\bf s}_{v}}\right)\right\} _{v\in\mathfrak{T}_{{\bf s}_{0}^{'}}}
\]
for $V$, we say $\mathcal{T}_{1}$ and $\mathcal{T}_{2}$ are \emph{non-equivalent
atlases} if for any $w$ in $\mathfrak{T}_{{\bf s}_{0}}$, there is
no $v$ in $\mathfrak{T}_{{\bf s}_{0}^{'}}$ such that the embeddings
$\iota_{{\bf s}_{w}}:T_{L,{\bf s}_{w}}\hookrightarrow V$ and $\iota_{{\bf s}_{v}}:T_{L,{\bf s}_{v}}\hookrightarrow V$
have the same image in $V$. Given two non-equivalent atlases of cluster
torus charts for $V$, we say they provide \emph{non-equivalent cluster
structures} for $V$ if both atlases cover $V$ up to codimension
$2$.
\end{defn}

Now, we are ready to apply the folding technique to families of log
Calabi-Yau surfaces. For $D^{\perp}\simeq\mathbf{D}_{n}\,(n\geq4)$,
let $\Pi$ be the subgroup of the symmetric group $S_{n+2}$ generated
by involutions
\[
(1\,\,2),\,\,\,(3\,\,4),\,\,\,(5\,\,6),(6\,\,7),\cdots,(n+1\,\,n+2).
\]
For $D^{\perp}\simeq\mathbf{E}_{n}\,(n=6,7,8)$, let $\Pi$ be the
subgroup of the symmetric group $S_{n+2}$ generated by involutions
\[
(1\,\,2),\,\,\,(3\,\,4),(4\,5),\,\,\,(6\,\,7),(7\,\,8),\cdots,(n+1\,\,n+2)
\]
Let $\overline{{\bf s}}$ be the folded seed via the symmetry of $\Pi$
on ${\bf s}$, then for $D^{\perp}\simeq\mathbf{D}_{n}\,(n\geq4)$,
$\overline{{\bf s}}$ has the exchange matrix equal to
\[
\begin{pmatrix}0 & -2 & 2\\
2 & 0 & -2\\
-\left(n-2\right) & n-2 & 0
\end{pmatrix}
\]
and for $D^{\perp}\simeq\mathbf{E}_{n}\,(n=6,7,8)$, $\overline{{\bf s}}$
has the exchange matrix equal to 
\[
\begin{pmatrix}0 & -2 & 2\\
3 & 0 & -3\\
-\left(n-3\right) & n-3 & 0
\end{pmatrix}.
\]
The $\mathcal{X}$-cluster variety $\mathcal{X}_{\overline{{\bf s}}}$
is a locally closed cluster subvariety of $\mathcal{X}_{{\bf s}}$.
By subsection \ref{subsec:Folding-and-degenerate}, $\mathcal{X}_{\overline{{\bf s}}}$
is a maximal degenerate $1$-parameter subfamily of $\mathcal{X}_{{\bf s}}$.

As we have seen in subsection \ref{subsec:Factorizations-of-DT-transformat}
and subsection \ref{subsec:Factorizations-of-DT-transformat-1}, there
exists a special element ${\bf w}$ of order $2$ in the Weyl group
of $(Y,D)$ that either coincides with the Donaldson-Thomas transformation
of $\mathcal{X}_{{\bf s}}\simeq\mathcal{X}_{D^{\perp}}$ or closely
relates to it. For $D^{\perp}\simeq\mathbf{D}_{n}\,$ ($n\geq4$),
${\bf w}$ is given by \ref{eq:-3 D_n}; for $D^{\perp}\simeq\mathbf{E}_{6},$
$\mathbf{E}_{7}$, and $\mathbf{E}_{8}$, ${\bf w}$ is given by \ref{eq: E_6},
\ref{eq:E_7} and \ref{eq:-3 E_8} respectively.
\begin{thm}
\label{thm:The-scattering-diagram}The scattering diagram $\mathfrak{D}_{\overline{{\bf s}}}$
has two distinct subfans corresponding to $\Delta_{\overline{{\bf s}}}^{+}$
and $\Delta_{\overline{{\bf s}}}^{-}$ respectively. The action of
Weyl group element ${\bf w}$ on $\mathfrak{D}_{{\bf s}}$ descends
to $\mathfrak{D}_{\overline{{\bf s}}}$ and interchanges $\Delta_{\overline{{\bf s}}}^{+}$
and $\Delta_{\overline{{\bf s}}}^{-}$. In particular, if we build
a variety $\mathcal{\tilde{A}}_{{\rm prin},\bar{{\bf s}}}$ using
both $\Delta_{\overline{{\bf s}}}^{+}$ and $\Delta_{\overline{{\bf s}}}^{-}$
as in Theorem 1.7 of \cite{YZ}, then $\Delta_{\overline{{\bf s}}}^{+}$
and $\Delta_{\overline{{\bf s}}}^{-}$ provide non-equivalent atlases
of cluster torus charts for $\mathcal{\tilde{A}}_{{\rm prin},\bar{{\bf s}}}$.
\end{thm}

\begin{proof}
Let $\mathcal{C}_{w}^{+}$ the cluster chamber that $\mathcal{C}_{{\bf s}}^{+}$
is mapped to under the action of ${\bf w}$ on $\mathfrak{D}_{{\bf s}}$.
As we have seen in subsection \ref{subsec:Factorizations-of-DT-transformat}
and subsection \ref{subsec:Factorizations-of-DT-transformat-1}, the
cluster chamber $\mathcal{C}_{w}^{+}$ is $\Pi$-invariant. By Lemma
\ref{lem:Pi-invariant chambers}, $\mathcal{C}_{w}^{+}$ corresponds
to a chamber $\overline{\mathcal{C}_{w}^{+}}$ in $\mathfrak{D}_{\overline{{\bf s}}}$.
Since $\mathcal{C}_{w}^{+}$ and $\mathcal{C}_{{\bf s}}^{-}$ are
related by $\Pi$-constrained mutations, $\overline{\mathcal{C}_{w}^{+}}$
and $\mathcal{C}_{\overline{{\bf s}}}^{-}$ belong to the same subfan
$\overline{\Delta}_{w\in{\bf s}}$ which coincides $\Delta_{\overline{{\bf s}}}^{-}$.
Let $U=Y\setminus D$. Then under the set-theoretic identification
of $\mathcal{X}_{{\bf s}}(\mathbb{R}^{T})$ with $M_{\mathbb{R}}$
as given by the seed ${\bf s}$, $U^{{\rm trop}}(\mathbb{R})$ can
be identified with the two-dimensional subspace of $M_{\mathbb{R}}$
defined by the equations 
\begin{align*}
 & \left\langle e_{1}-e_{2},\cdot\right\rangle =0,\left\langle e_{3}-e_{4},\cdot\right\rangle =0,\\
 & \left\langle e_{5}-e_{6},\cdot\right\rangle =\left\langle e_{6}-e_{7},\cdot\right\rangle \cdots=\left\langle e_{n+1}-e_{n+2},\cdot\right\rangle =0,\\
 & \left\langle e_{1}+e_{3}+e_{5},\cdot\right\rangle =0
\end{align*}
if $D^{\perp}\simeq\mathbf{D}_{n}$ ($n\geq4)$ and by the equations
\begin{align*}
 & \left\langle e_{1}-e_{2},\cdot\right\rangle =0,\left\langle e_{3}-e_{4},\cdot\right\rangle =\left\langle e_{4}-e_{5},\cdot\right\rangle =0,\\
 & \left\langle e_{6}-e_{7},\cdot\right\rangle =\left\langle e_{7}-e_{8},\cdot\right\rangle \cdots=\left\langle e_{n+1}-e_{n+2},\cdot\right\rangle =0,\\
 & \left\langle e_{1}+e_{3}+e_{6},\cdot\right\rangle =0
\end{align*}
if $D^{\perp}\simeq\mathbf{E}_{n}$ ($n=6,7,8)$. Thus via the embedding
$q^{*}:\overline{M}_{\mathbb{R}}\hookrightarrow M_{\mathbb{R}},$
$U^{{\rm trop}}(\mathbb{R})$ can be identified as the hyperplane
in $\overline{M}_{\mathbb{R}}$ given by the equation 
\[
\left\langle e_{\Pi1}+e_{\Pi3}+e_{\Pi5},\cdot\right\rangle =0
\]
if $D^{\perp}\simeq\mathbf{D}_{n}$ ($n\geq4)$ and by the equation
\[
\left\langle e_{\Pi1}+e_{\Pi3}+e_{\Pi6},\cdot\right\rangle =0
\]
if $D^{\perp}\simeq\mathbf{E}_{n}$ ($n=6,7,8)$. By Theorem 4.5 of
\cite{TM}, $U^{{\rm trop}}(\mathbb{R})$ intersects with $\Delta_{{\bf s}}^{+}$
only at the origin. Thus, $U^{{\rm trop}}(\mathbb{R})$ intersects
with $\Delta_{\overline{{\bf s}}}^{+}$ and $\overline{\Delta}_{w\in{\bf s}}=\Delta_{\overline{{\bf s}}}^{-}$
only at the origin. Since the chamber $\mathcal{C}_{\overline{{\bf s}}}^{+}$
of $\Delta_{\overline{{\bf s}}}^{+}$ and the chamber $\mathcal{C}_{\overline{{\bf s}}}^{-}$
of $\Delta_{\overline{{\bf s}}}^{-}$ live in different half spaces
separated by $U^{{\rm trop}}(\mathbb{R})$ as a hyperplane in $\overline{M}_{\mathbb{R}}$,
we conclude that $\Delta_{\overline{{\bf s}}}^{+}$ and $\Delta_{\overline{{\bf s}}}^{-}$
are two distinct subfans in $\mathfrak{D}_{\overline{{\bf s}}}$ .

It follows from Proposition 3.4 and Proposition 3.6 of \cite{YZ}
that the action of Weyl group element ${\bf w}$ on $\mathfrak{D}_{{\bf s}}$
descends to $\mathfrak{D}_{\overline{{\bf s}}}$. Since the induced
action of ${\bf w}$ on $\mathfrak{D}_{\overline{{\bf s}}}$ maps
$\Delta_{{\bf \overline{s}}}^{+}$ to $\Delta_{\overline{{\bf s}}}^{-}$
and ${\bf w}$ is of order $2$, we conclude that it interchanges
$\Delta_{\overline{{\bf s}}}^{+}$ and $\Delta_{\overline{{\bf s}}}^{-}$.
The last statement of the theorem follows from Theorem 1.7 of \cite{YZ}. 
\end{proof}
Let us end by the final remark that there are other, potentially more
interesting, ways to apply the folding procedure in positive non-acyclic
cases. For example when $D^{\perp}\simeq\mathbf{D}_{2m}$ ($m\geq2$),
for each $1\leq l\leq m-1$, denote by $\alpha_{l}$ the root $2l-E_{1}-E_{2}-E_{3}-E_{4}-E_{2l+3}-E_{2l+4}$.
Observe that we have the following relation in the Weyl group
\[
r_{\alpha_{l}}={\bf w}_{2l+3}\circ{\bf w}_{2l+4}
\]
where ${\bf w}_{i}$ is defined as in \ref{eq:-3_w_i}. Let $\mathcal{P}$
be the power set of $\{1,\cdots l\}$. Given $\varpi=\{i_{1},\cdots,i_{h}\}\in\mathcal{P}$,
let ${\bf s}_{\varpi}$ be the seed after performing the sequence
of reflections 
\[
r_{\alpha_{i_{1}}}\circ\cdots\circ r_{\alpha_{i_{h}}}.
\]
Since the subgroup of the Weyl group generated by $r_{\alpha_{l}}$
is abelian, $\mathcal{C}_{{\bf s}_{\varpi}}^{+}$ is independent of
the choice of ordering of elements in $\varpi$. We have the following
conjecture:
\begin{conj}
 \label{thm:conj1}
For any two distinct elements $\varpi,\varpi^{'}$ in $\mathcal{P}$, $\overline{\Delta}_{{\bf s}_{\varpi},\Pi}^{+}$ and $\overline{\Delta}_{{\bf s}_{\varpi^{'}},\Pi}^{+}$ has trivial intersection.
\end{conj}The above conjecture is based on the numerical evidence that the chambers
in $\Delta_{{\bf s},\Pi}^{+}$ is contained in the region in $\mathcal{X}_{{\bf D}_{2m}}^{{\rm trop}}$
cut out by the equation
\[
\left(z^{e_{1}+e_{3}+e_{5}}\right)^{T}\geq0,\left(z^{e_{1}+e_{3}+e_{7}}\right)^{T}\geq0,\cdots,\left(z^{e_{1}+e_{3}+e_{2m+1}}\right)^{T}\geq0
\]
while given any nonempty subset $\varpi$ in $\mathcal{P}$, $\mathcal{C}_{{\bf s}_{\varpi}}^{+}$
is outside the above region.

\appendix

\section{Mutation sequences for ${\rm DT}_{\mathcal{X}_{\mathbf{E}_{n}}}\,(n=6,7,8)$
in SAGE code. \label{sec:Mutation-sequences-for}}

In the ClusterSeed package of SAGE, the indexing for cluster seeds starts from $0$ instead of $1$. Thus we need to shift mutation sequences given in the paper by $-1$ in SAGE code.

\subsection{\label{subsec:A1}Mutation sequence for ${\rm DT}_{\mathcal{X}_{\mathbf{E}_{6}}}$.}
In the following SAGE code, the first three lines of the mutation sequence give reflections $r_{\alpha_{i}}(i=1,2,3)$ as defined in \ref{eq:-3}. The two mutations $[0,1]$ in the last line of the mutation
sequence is the difference between the Weyl group element ${\bf w}=r_{\alpha_{1}}\circ r_{\alpha_{2}}\circ r_{\alpha_{3}}$
and ${\rm DT}_{\mathcal{X}_{\mathbf{E}_{6}}}$.

\begin{sageexample}
sage: S = ClusterSeed(matrix([[0,0,-1,-1,-1,1,1,1],[0,0,-1,-1,-1,1,1,1],
 ....: [1,1,0,0,0,-1,-1,-1],[1,1,0,0,0,-1,-1,-1],[1,1,0,0,0,-1,-1,-1],
 ....: [-1,-1,1,1,1,0,0,0],[-1,-1,1,1,1,0,0,0],[-1,-1,1,1,1,0,0,0]]));
sage: S.mutate([0,2,5,3,6,0,1,2,5,3,6,1]);
sage: S.mutate([0,2,5,4,7,0,1,2,5,4,7,1]);
sage: S.mutate([0,3,6,4,7,0,1,3,6,4,7,1]);
sage: S.g_matrix()
sage: S.mutate([0,1]); 
sage: S.g_matrix()
\end{sageexample}
\\
An alternative, shorter mutation sequence for ${\rm DT}_{\mathbf{E}_{6}}$ is as follows:
\begin{sageexample}
sage: S = ClusterSeed(matrix([[0,0,-1,-1,-1,1,1,1],[0,0,-1,-1,-1,1,1,1],
 ....: [1,1,0,0,0,-1,-1,-1],[1,1,0,0,0,-1,-1,-1],[1,1,0,0,0,-1,-1,-1],
 ....: [-1,-1,1,1,1,0,0,0],[-1,-1,1,1,1,0,0,0],[-1,-1,1,1,1,0,0,0]]));
sage: S.mutate([0,2,5,3,6,4,7,0,1,2,5,3,6,4,7,1,0,2,5,3,6,4,7,1]);
sage: S.g_matrix()
\end{sageexample} 
\\
\subsection{Mutation sequence for ${\rm DT}_{\mathcal{X}_{\mathbf{E}_{7}}}$.}

In the following SAGE code, each line of the mutation sequence gives
one of the reflections $r_{\alpha_{i}}(i=1,2,\cdots,6)$ as defined
in \ref{eq:-4}. In this case, the Weyl group element ${\bf w=}\prod_{i=1}^{6}r_{\alpha_{i}}$
agrees with ${\rm DT}_{\mathcal{X}_{\mathbf{E}_{7}}}$.

\begin{sageexample}
sage: S = ClusterSeed(matrix([[0,0,-1,-1,-1,1,1,1,1],[0,0,-1,-1,-1,1,1,1,1],
....: [1,1,0,0,0,-1,-1,-1,-1],[1,1,0,0,0,-1,-1,-1,-1],[1,1,0,0,0,-1,-1,-1,-1],
....: [-1,-1,1,1,1,0,0,0,0],[-1,-1,1,1,1,0,0,0,0],[-1,-1,1,1,1,0,0,0,0],
....: [-1,-1,1,1,1,0,0,0,0]]))
sage: S.use_fpolys(False);
sage: S.mutate([0,2,3,5,6,0,1,2,3,5,6,1]);
sage: S.mutate([0,2,4,5,7,0,1,2,4,5,7,1]);
sage: S.mutate([0,3,4,5,8,0,1,3,4,5,8,1]);
sage: S.mutate([0,2,3,7,8,0,1,2,3,7,8,1]);
sage: S.mutate([0,2,4,6,8,0,1,2,4,6,8,1]);
sage: S.mutate([0,3,4,6,7,0,1,3,4,6,7,1]);
sage: S.g_matrix()
\end{sageexample}
\\
\subsection{Mutation sequence for ${\rm DT}_{\mathcal{X}_{\mathbf{E}_{8}}}$.}

In the following SAGE code, the first four lines of the mutation sequence
give the reflection $r_{\alpha}$ as defined in \ref{eq:-5}. The
next six lines give reflections $r_{\beta_{i}}(i=1,2,\cdots,6)$ as
defined in \ref{eq:-6}. In this case, the Weyl group element ${\bf w}=r_{\alpha}\circ\prod_{i=1}^{6}r_{\beta_{i}}$
also agrees with ${\rm DT}_{\mathcal{X}_{\mathbf{E}_{8}}}$. 

\begin{sageexample}
sage: S = ClusterSeed(matrix([[0,0,-1,-1,-1,1,1,1,1,1],[0,0,-1,-1,-1,1,1,1,1,1],
....: [1,1,0,0,0,-1,-1,-1,-1,-1],[1,1,0,0,0,-1,-1,-1,-1,-1],[1,1,0,0,0,-1,-1,-1,-1,-1],
....: [-1,-1,1,1,1,0,0,0,0,0],[-1,-1,1,1,1,0,0,0,0,0],[-1,-1,1,1,1,0,0,0,0,0],
....: [-1,-1,1,1,1,0,0,0,0,0],[-1,-1,1,1,1,0,0,0,0,0]]))
sage: S.use_fpolys(False);
sage: S.mutate([0,4,7,0]);S.mutate([0,2,3,5,6,0,1,2,3,5,6,1]);S.mutate([1,4,7,1]);
sage: S.mutate([1,4,9,1]);S.mutate([0,2,3,5,8,0,1,2,3,5,8,1]);S.mutate([0,4,9,0]);
sage: S.mutate([1,4,7,1]);S.mutate([0,2,3,5,6,0,1,2,3,5,6,1]);S.mutate([0,4,7,0]); 
sage: S.mutate([0,4,9,0]);S.mutate([0,2,3,5,8,0,1,2,3,5,8,1]);S.mutate([1,4,9,1]);
sage: S.mutate([0,2,3,6,7,0,1,2,3,6,7,1]);
sage: S.mutate([0,3,4,6,8,0,1,3,4,6,8,1]);
sage: S.mutate([0,2,4,6,9,0,1,2,4,6,9,1]);
sage: S.mutate([0,2,3,8,9,0,1,2,3,8,9,1]);
sage: S.mutate([0,3,4,7,9,0,1,3,4,7,9,1]);
sage: S.mutate([0,2,4,7,8,0,1,2,4,7,8,1]);
sage: S.g_matrix()
\end{sageexample}

\end{document}